\documentclass[12pt]{amsart}
\usepackage{enumerate,amssymb,amsfonts,latexsym,psfrag,epsfig}
\usepackage{amscd,amsmath}
\usepackage{fullpage}
\newtheorem{thm}{Theorem}[section]
\newtheorem{lem}[thm]{Lemma}
\newtheorem{cor}[thm]{Corollary}

\newtheorem{prop}[thm]{Proposition}

\newtheorem{defn}{Definition}

\newtheorem{rem}[thm]{Remark}

\begin{document}

\bigskip\bigskip
\title[Chaotic Period Doubling] {Chaotic Period Doubling}

\author{V.V.M.S. Chandramouli, M. Martens, \\
W. de Melo, C.P. Tresser.}

\address{University of Groningen, The Netherlands}
\address{Suny at Stony Brook, USA}
\address{IMPA, Brazil}
\address{IBM, USA}

\date{September 18, 2007}
\begin{abstract} The period doubling renormalization operator was introduced by M. Feigenbaum and by P. Coullet and C.
Tresser in the nineteen-seventieth to study the asymptotic small
scale geometry of the attractor of one-dimensional systems which are
at the transition from simple to chaotic dynamics.  This geometry
turns out to not depend on the choice of the map under rather mild
smoothness conditions. The existence of a unique renormalization
fixed point which is also hyperbolic among generic smooth enough
maps plays a crucial role in the corresponding renormalization
theory. The uniqueness and hyperbolicity of the renormalization
fixed point were first shown in the holomorphic context, by means that 
generalize to other renormalization operators.  It was
then proved that in the space of $C^{2+\alpha}$ unimodal maps,
for $\alpha$ close to one, the period doubling renormalization fixed
point is hyperbolic as well.  In this paper we study what happens when one
approaches from below the minimal smoothness thresholds for the
uniqueness and for the hyperbolicity of the period doubling
renormalization generic fixed point.  Indeed, our main results
states that in the space of $C^2$ unimodal maps the analytic fixed
point is not hyperbolic and that the same remains true when adding
enough smoothness to get a priori bounds. In this smoother class,
called $C^{2+|\cdot|}$ the failure of hyperbolicity is tamer than in
$C^2$. Things get much
 worse with just a bit less of smoothness than $C^2$ as then even the 
uniqueness is lost and other asymptotic behavior become possible. We show
  that the period doubling renormalization operator acting on the space of 
$C^{1+Lip}$ unimodal maps has infinite topological entropy.
\end{abstract}

\maketitle
\thispagestyle{empty}
\def\IMSmarkvadjust{0 pt}
\def\IMSmarkhadjust{0 pt}
\def\IMSmarkhpadding{0 pt}
\def\IMSpubltext{Published in modified form:}
\def\SBIMSMark#1#2#3{
 \font\SBF=cmss10 at 10 true pt
 \font\SBI=cmssi10 at 10 true pt
 \setbox0=\hbox{\SBF \hbox to \IMSmarkhpadding{\relax}
                Stony Brook IMS Preprint \##1}
 \setbox2=\hbox to \wd0{\hfil \SBI #2}
 \setbox4=\hbox to \wd0{\hfil \SBI #3}
 \setbox6=\hbox to \wd0{\hss
             \vbox{\hsize=\wd0 \parskip=0pt \baselineskip=10 true pt
                   \copy0 \break%
                   \copy2 \break%
                   \copy4 \break}}
 \dimen0=\ht6   \advance\dimen0 by \vsize \advance\dimen0 by 8 true pt
                \advance\dimen0 by -\pagetotal
	        \advance\dimen0 by \IMSmarkvadjust
 \dimen2=\hsize \advance\dimen2 by .25 true in
	        \advance\dimen2 by \IMSmarkhadjust

%
%
  \openin2=publishd.tex
  \ifeof2\setbox0=\hbox to 0pt{}
  \else 
     \setbox0=\hbox to 3.1 true in{
                \vbox to \ht6{\hsize=3 true in \parskip=0pt  \noindent  
                {\SBI \IMSpubltext}\hfil\break
                \input publishd.tex 
                \vfill}}
  \fi
  \closein2
  \ht0=0pt \dp0=0pt
 \ht6=0pt \dp6=0pt
 \setbox8=\vbox to \dimen0{\vfill \hbox to \dimen2{\copy0 \hss \copy6}}
 \ht8=0pt \dp8=0pt \wd8=0pt
 \copy8
 \message{*** Stony Brook IMS Preprint #1, #2. #3 ***}
}

\SBIMSMark{2007/2}{September 2007}{}


\setcounter{tocdepth}{1}
\tableofcontents

\section{Introduction}

The period doubling renormalization operator was introduced by M.
Feigenbaum \cite{Fe}, \cite{Fe2} and by P. Coullet and C. Tresser
\cite{CT}, \cite{TC} to study the asymptotic small scale geometry of
the attractor of one-dimensional systems which are at the transition
from simple to chaotic dynamics.  In 1978, they published certain
rigidity properties of such systems, the small scale geometry of the
invariant Cantor set of generic smooth maps at the boundary of chaos
being independent of the particular map being considered. Coullet
and Tresser treated this phenomenon as similar to
\emph{universality} that has been observed in critical phenomena for
long and explained since the early seventieth by Kenneth Wilson
(see, \emph{e.g.,} \cite{Ma}).  In an attempt to explain
universality at the transition to chaos, both groups formulated the
following conjectures that are similar to what was conjectured in
statistical mechanics.

\medskip
\noindent Renormalization conjectures: \emph{In the proper class of
maps, the period doubling \emph{renormalization operator} has a
unique fixed point that is hyperbolic with a one-dimensional
unstable manifold and a codimension one stable manifold consisting
of the systems at the transition to chaos.}

\smallskip
These conjectures were extended to other types of dynamics on the
interval and on other manifolds but we will not be concerned here
with such generalizations. During the last 30 years many authors
have contributed to the development of a rigorous theory proving the
renormalization conjectures and explaining the phenomenology. The
ultimate goal may still be far since the universality class of
smooth maps at the boundary of chaos contains many sorts of
dynamical systems, including useful differential models of natural
phenomena and there even are predictions about natural phenomena in
\cite{CT}, which turned out to be experimentally corroborated. A
historical review of the mathematics that have been developed can be
found in \cite{FMP} so that we recall here only a few milestones
that will serve to better understand the contribution to the overall
picture brought by the present paper.

The type of differentiability of the systems under consideration has
a crucial influence on the actual small scale geometrical behavior
(like it is the case in the related problem of smooth conjugacy  of
circle diffeomorphisms to rotations: compare \cite{He} to \cite{KO}
and \cite{KS}). The first result dealt with holomorphic systems and
were first local \cite{La}, and later global \cite{Su}, \cite{McM},
\cite{Ly} (a progression similar to what had been seen in the
problem of smooth conjugacy to rotations: compare \cite{Ar} to
\cite{He} and \cite{Yo}). With global methods came also means to
consider other renormalizations.  Indeed, the hyperbolicity of the
unique renormalization fixed point has been shown in \cite{La} for
period doubling, and later in \cite{Ly} by means that generalize to
other sorts of dynamics.  Then it  was showed in \cite{Da} that the
renormalization fixed point is also hyperbolic in the space of
$C^{2+\alpha}$ unimodal maps with $\alpha$ close to one (using
\cite{La}), these results being later extended in \cite{FMP} to more
general types of renormalization  (using \cite{Ly}). As far as
existence of fixed points is concerned, a satisfactory theory could
be obtained some time ago, first for period doubling only and then
for maps with bounded combinatorics after several subclasses of
dynamics had been solved, see \cite{M} for the most general results,
assuming the lowest degree of smoothness and references to the prior
literature.

We are interested in exploring from below the limit of smoothness
that permits hyperbolicity of the fixed point of renormalization.
Our main result concern a new smoothness class, $C^{2+|\cdot|}$,
which is bigger than $C^{2+ \alpha}$ for any positive $\alpha\leq
1$, and is in fact wider than $C^{2}$ in ways that are rather
technical as we shall describe later (this is the bigger class where
the usual method to get a priori bounds for the geometry of the
Cantor set work).  We are interested here in the part of
hyperbolicity that consists in the attraction in the stable manifold
made of infinitely renomalizable maps.
We show that in the space of $C^{2+|\cdot|}$ unimodal maps the
analytic fixed point is not hyperbolic for the action of the period
doubling renormalization operator. We also show that nevertheless,
the renormalization converges to the analytic generic fixed point
(here generic means that the second derivative at the critical point
is not zero), proving it to be globally unique, a uniqueness that
was formerly known in classes smaller than $C^{2+|\cdot|}$
(hence assuming more smoothness). The convergence might only be
polynomial as a concrete sign of non-hyperbolicity. The failure of
hyperbolicity happens in a more serious way in the space of $C^2$
unimodal maps since there the convergence can be arbitrarily slow.
The uniqueness of the fixed point in this case, remains an open question. 
The uniqueness was known to be wrong in a serious way among $C^{1+Lip}$
unimodal maps since a continuum of fixed points of renormalization
could be produced \cite{Tr}. Here we show that the period doubling
renormalization operator acting on the space of $C^{1+Lip}$ unimodal
maps has infinite topological entropy.

After this informal discussion of what will be done here and how it 
relates to universality theory, we now give some definitions, which allows 
us next to turn to the precise formulation of our main results.

\bigskip

A {\it unimodal} map $f:[0,1]\to [0,1]$ is a $C^1$ mapping with the
following properties.
\begin{itemize}
\item $f(1)=0$, \item there is a unique point $c\in (0,1)$, the
{\it critical point}, where $Df(c)=0$, \item $f(c)=1$.
\end{itemize}
A map is a $C^r$ unimodal maps if $f$ is $C^r$. We will
concentrate on unimodal maps of the type $C^{1+Lip}$, $C^2$, and
$C^{2+|\cdot|}$. This last type of differentiability will be introduced
in $\S~\ref{c2abs}.$

The critical point $c$ of a $C^2$ unimodal map $f$ is called {\it
non-flat} if $D^2f(c)\ne 0$. A critical point $c$ of a unimodal
map $f$ is a {\it quadratic tip} if there exists a sequence of
points $x_n \to c$ and constant $A >0$ such that
$$ \lim_{n \to \infty} \frac{f(x_n)- f(c)}{(x_n-c)^2} = -A .$$
The set of $C^r$ unimodal maps with a quadratic tip is denoted by
$\mathcal{U}^r$. We will consider different metrics on this set
denoted by $\text{dist}_k$ with $k=0,1,2$ (in fact the usual 
$C^k$ metrics).

A unimodal map $f:[0,1]\to [0,1]$ with quadratic tip $c$ is {\it
renormalizable} if
\begin{itemize}
\item $c\in [f^2(c), f^4(c)]\equiv I_0^1$, \item
$f(I_0^1)=[f^3(c),f(c)]\equiv I_1^1$, \item $I_0^1\cap
I_1^1=\emptyset$.
\end{itemize}
The set of renormalizable $C^r$ unimodal maps is denoted by
$\mathcal{U}_0^r\subset \mathcal{U}^r$. Let $f\in \mathcal{U}_0^r$
be a renormalizable map. The renormalization of $f$ is defined by
$$
Rf(x) =   h^{-1}    \circ  f^2\circ h(x),
$$
where $h:[0,1]\to I_0^1$ is the orientation reversing affine
homeomorphism. This map $Rf$ is again a unimodal map. The
nonlinear operator $R:\mathcal{U}_0^r\to \mathcal{U}^r$ defined by
$$
R:f\mapsto Rf
$$
is called the {\it renormalization operator}. The set of {\it
infinitely} renormalizable maps is denoted by
$$
W^r=\bigcap_{n\ge 1} R^{-n}(\mathcal{U}_0^r).
$$
There are many fundamental steps needed to reach the following
result by Davie, see \cite{Da}. For a brief history see \cite{FMP}
and references therein.

\begin{thm}\label{davie} (Davie) Let $\alpha<1$ close enough to one. There exists a
unique renormalization fixed point $f_*^\omega\in
\mathcal{U}^{2+\alpha}$. It has the following properties.
\begin{itemize}
\item $f_*^\omega$ is analytic, \item $f_*^\omega$ is a hyperbolic
fixed point of $R:\mathcal{U}_0^{2+\alpha}\to
\mathcal{U}^{2+\alpha} $, \item the codimension one stable
manifold of $f_*^\omega$ coincides with $W^{2+\alpha}$. \item
$f_*^\omega$ has a one dimensional unstable manifold which
consists of analytic maps.
\end{itemize}
\end{thm}

In our discussion we only deal with period doubling
renormalization. However, there are other renormalization schemes.
The hyperbolicity for the corresponding generalized
renormalization operator has been established in \cite{FMP}.\\

Our main results deal with $R:\mathcal{U}_0^{r}\to
\mathcal{U}^{r}$ where  $r\in\{1+Lip, 2, 2+|\cdot|\}$.

\begin{thm}
Let $d_n > 0$ be any sequence with $d_n \to 0$. There exists an
infinitely renormalizable $C^2$ unimodal map $f$ with quadratic
tip such that
$$dist_0 \left( R^nf, f^{\omega}_{*} \right)  \geq d_n.$$
\end{thm}

\begin{cor} The analytic unimodal map $f_*^\omega$
is not a hyperbolic fixed point of $R:\mathcal{U}_0^{2}\to
\mathcal{U}^{2} $.
\end{cor}

In $\S~\ref{c2abs}$ we will introduce a type of
differentiability of a unimodal map, called $C^{2+|\cdot|}$,  which is
the minimal needed to be able to apply the classical proofs of a
priori bounds for the invariant Cantor sets of infinitely
renormalizable maps, see for example
\cite{M2},\cite{MMSS},\cite{MS}. This type of differentiability
will allow us to represent any $C^{2+|\cdot|}$ unimodal map as
$$
f=\phi \circ q,
$$
where $q$ is a quadratic polynomial and $\phi$ has still enough
differentiability to control cross-ratio distortion. The precise
description of this decomposition is given in Proposition
\ref{non-lin-bound}. For completeness we include the proof of the
a priori bounds in $\S~\ref{ap bounds}.$

\begin{thm}
If $f$ is an infinitely renormalizable $C^{2+|\cdot|}$ unimodal map then
$$\lim_{n \to \infty} dist_0 \left( R^nf, \;f_{*}^{\omega} \right) = 0.$$
\end{thm}

A construction similar to the one provided for $C^2$ unimodal maps leads
to the following result:

\begin{thm} Let $d_n
> 0$ be any sequence with $\sum_{n\ge 1} d_n < \infty$. There exists an infinitely
renormalizable $C^{2+|\cdot|}$ unimodal map $f$ with a quadratic tip
such that
$$dist_0 \left( R^nf, f^{\omega}_{*} \right)  \geq d_n.$$

The analytic unimodal map $f_*^\omega$ is not a hyperbolic fixed
point of $R:\mathcal{U}_0^{2+|\cdot|}\to \mathcal{U}^{2+|\cdot|} $.
\end{thm}

Our second set of theorems deals with renormalization of
$C^{1+Lip}$ unimodal maps with a quadratic tip.

\begin{thm}
There exists an infinitely renormalizable $C^{1+Lip}$ unimodal map
$f$ with a quadratic tip which is not $C^2$ but
$$Rf = f.$$
\end{thm}

The {\it topological entropy} of a system defined on a noncompact space is 
defined to be the supremum of the topological entropies contained in 
compact invariant subsets: we will always mean topological entropy when 
the type of entropy is not specified. As a consequence of Theorem 
~\ref{davie} we get that renormalization on
 $\mathcal{U}_0^{2+\alpha}$ has entropy zero.

\begin{thm} The
renormalization operator acting on the space of
$C^{1+Lip}$ unimodal maps with quadratic tip has infinite entropy.
\end{thm}

The last
theorem illustrates a specific aspect of the chaotic behavior of the
renormalization operator on $\mathcal{U}_0^{1+Lip}$:

\begin{thm}
There exists an infinitely renormalizable $C^{1+Lip}$ unimodal map
$f$ with quadratic tip such that $\displaystyle{\left \{c_n \right
\}_{n \geq 0}}$ is dense in a Cantor set. Here $c_n$ is the
critical point of $R^nf$.
\end{thm}

\noindent
{\bf Acknowledgement} W.de Melo was partially supported by 
CNPq-304912/2003-4 and FAPERJ E-26/152.189/2002.

\section{Notation}

Let $I,J\subset \mathbb{R}^n$, with $n\ge 1$. We will use the
following notation.
\begin{itemize}
\item $cl(I)$, $int(J),$ $\partial{I},$ stands for resp. the
closure, the interior, and the boundary of $I$.
\item $|I|$ stands
for the Lebesgue measure of $I$.
 \item If $n=1$ then $[I, J]$ is smallest interval which contains $I$ and $J$.
\item $dist \;(x,y)$
is the Euclidean distance between $x$ and $y$, and $$dist \;(I,J)=\inf 
_{x\in I,\,y\in J}dist \;(x,y).$$
 \item If $F$ is a map between two sets then $\text{image}(F)$
 stand for the image of $F$.
\item Define $\text{Diff}_+^k\;([0, 1])$, $k\ge 1$, is the set of orientation preserving $C^k-$diffeomorphisms.
\item $|.|_{k}$, $k\ge 0$, stands for the $C^k$ norm of the
functions under consideration.
\item $dist_k$, $k\ge 0$, stands for the $C^k$ distance in the
function spaces under consideration.

\item There is a constant $K>0$, held fixed throughout the paper, which lets us write $Q_1\asymp Q_2$ if and only if
$$
\frac1K\le \frac{Q_1}{Q_2}\le K.
$$
\end{itemize}

There are two rather independent discussions. One on $C^{1+Lip}$maps
and the other on $C^2$ maps. There is a slight conflict in the
notation used for these two discussions. In particular, the notation
$I^n_1$ stands for different intervals in the two parts, but the
context will make the meaning of the symbols unambiguous.

\section{Renormalization of $C^{1+Lip}$ unimodal maps} \label{ren-of-c1lip}

\subsection{Piece-wise affine infinitely renormalizable maps.}
\noindent
Consider the open triangle $\Delta = \{(x, y) \;:\; x, y > 0\; \text{and} \;x+y < 1 \}.$
A point $(\sigma_0,\sigma_1) \in \Delta$ is called a {\it scaling bi-factor}.
A scaling bi-factor induces a pair of affine maps
$$ \tilde{\sigma}_0: [0, 1] \rightarrow [0, 1]\,,$$
$$ \tilde{\sigma}_1: [0, 1] \rightarrow [0, 1]\,,$$
defined by
\begin{eqnarray*}
 \tilde{\sigma}_0(t)  & = &  -\sigma_0 t+\sigma_0  =  \sigma_0(1-t)\\
 \tilde{\sigma}_1(t) & = &  \sigma_1 t+ 1-\sigma_1 =  1-\sigma_1(1-t).
\end{eqnarray*}
A function $\sigma : \mathbb{N} \rightarrow \Delta$ is called a {\it scaling data}.
For each $n\in \mathbb{N}$ we set $\sigma(n) = \left(\sigma_0(n), \sigma_1(n)\right)$, so that the point
$(\sigma_0(n), \sigma_1(n)\in \Delta$ induces a pair of maps $( 
\tilde{\sigma}_0(n),  \tilde{\sigma}_1(n)$ as we have just described. For 
each $n\in \mathbb{N}$ we can now define the pair of intervals:
$$ I_0^n =  \tilde{\sigma}_0(1) \circ  \tilde{\sigma}_0(2) \circ \dots  \circ  \tilde{\sigma}_0(n)([0, 1])\,,$$
$$ I_1^n =  \tilde{\sigma}_0(1) \circ  \tilde{\sigma}_0(2) \circ \dots  \circ  \tilde{\sigma}_0(n-1) \circ
 \tilde{\sigma}_1(n)([0, 1])\,.$$

\begin{figure}[ht]
\centering
\psfrag{I01}[c][c][1][0]{${\rm I}_0^1$}
\psfrag{Io1}[c][c][1][0]{${\rm I}_1^1$}
\psfrag{I02}[c][c][1][0]{${\rm I}_0^2$}
\psfrag{I03}[c][c][1][0]{${\rm I}_0^3$}
\psfrag{I13}[c][c][1][0]{${\rm I}_1^3$}
\psfrag{I12}[c][c][1][0]{${\rm I}_1^2$}
\psfrag{C}[c][c][1][0]{$c$}
\includegraphics[scale=1.5]{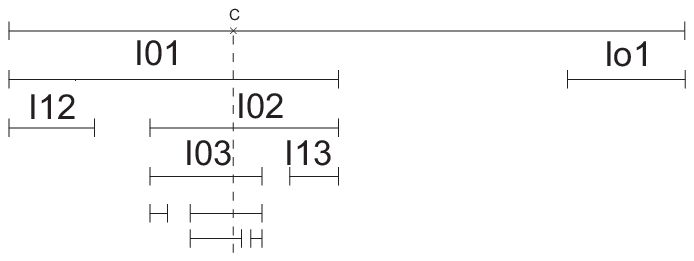}
\caption{}
\label{cantorset}
\end{figure}


A scaling data with the property
$${dist} \left(\sigma(n), \partial{\Delta} \right) \geq \epsilon >0$$
is called $\epsilon-${\it proper}, and {\it proper} if it is $\epsilon-$proper for some $\epsilon >0$. For $\epsilon-${\it proper} scaling data we have
$$ |I_j^n| \leq (1-\epsilon)^n$$
$\text{with} \;n \geq 1 \;\text{and} \; j=0,1.$
Given proper scaling data define
$$ \{c\} = \cap_{n \geq 1} I_0^n.$$
The point $c$, called the {\it critical} point, is shown in
Figure $\ref{cantorset}$. Consider the quadratic map 
$q_c : [0, 1] \rightarrow [0, 1]$ defined as: 
$$q_c(x) = 1- \left(\frac{x-c}{1-c} \right)^2.$$

\begin{figure}[ht]
\centering \psfrag{I01}[c][c][1][0]{${I}_0^1$}
\psfrag{Io1}[c][c][1][0]{${I}_1^1$}
\psfrag{I02}[c][c][1][0]{${I}_0^2$}
\psfrag{I03}[c][c][1][0]{${I}_0^3$}
\psfrag{I13}[c][c][1][0]{${I}_1^3$}
\psfrag{I12}[c][c][1][0]{${I}_1^2$}
\psfrag{qc}[c][c][1][0]{${q_c}$}
\psfrag{C}[c][c][1][0]{$c$}
\psfrag{fs}[c][c][1][0]{$f_{\sigma}$}
\includegraphics[scale=1.15]{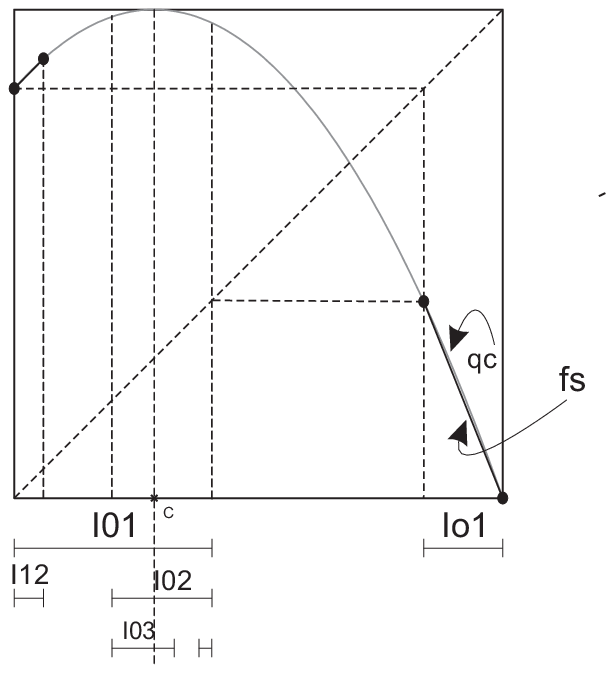}
\caption{}
\label{test.eps}
\end{figure}


Given a proper scaling data
$\sigma : \mathbb{N} \rightarrow \Delta $ and the set 
$D_{\sigma} = \cup_{n \geq 1} I_1^n$ induced by $\sigma$, 
we define a map $$f_{\sigma}: D_{\sigma} \rightarrow [0, 1]$$
by letting $f_{\sigma}|_{I_1^n}$ be the affine extension of 
$q_{c}|_{\partial{I_1^n}}.$ The graph of $f_{\sigma}$ is shown in 
Figure $\ref{test.eps}.$

\begin{figure}[ht]
\centering
\psfrag{I0n}[c][c][1][0]{${\rm I}_0^n$}
\psfrag{I1n}[c][c][1][0]{${\rm I}_1^n$}
\psfrag{I1n}[c][c][1][0]{${\rm I}_1^n$}
\psfrag{I0n-1}[c][c][1][0]{${\rm I}_0^{n-1}$}
\psfrag{Xn}[c][c][1][0]{${x_{n-1}}$}
\psfrag{xn1}[c][c][1][0]{${x_{n+1}}$}
\psfrag{Xn-1}[c][c][1][0]{${x_{n}}$}
\psfrag{Xn-2}[c][c][1][0]{${y_{n}}$}
\psfrag{Yn}[c][c][1][0]{${x_{n-2}}$}
\psfrag{yn1}[c][c][1][0]{${y_{n+1}}$}
\psfrag{C}[c][c][1][0]{${c}$}
\includegraphics[scale=0.95]{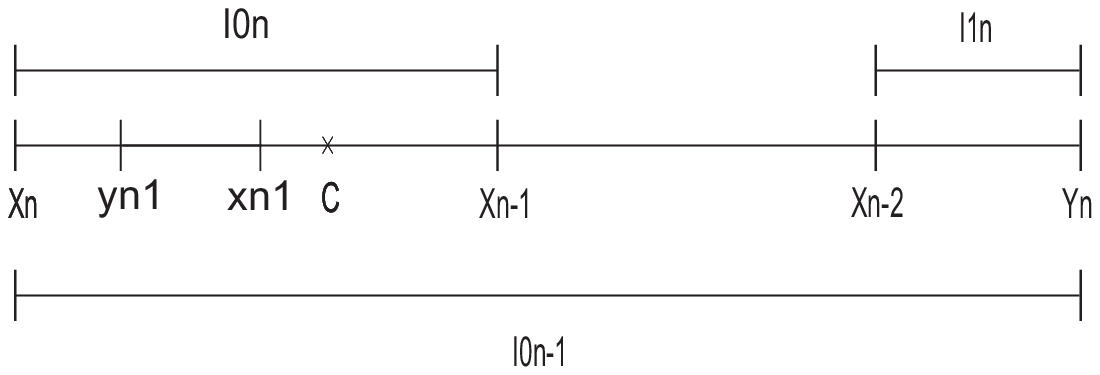}
\caption{}
\label{nthlevel}
\end{figure}


Define $x_0 =0, x_{-1}=1$ and for $n \geq 1$
\begin{eqnarray*}
x_n & = & \partial{I_0^n} \setminus \partial{I_0^{n-1}}, \\
y_n & = & \partial{I_1^n} \setminus \partial{I_0^{n-1}}.
\end{eqnarray*}
These points are illustrated in Figure $\ref{nthlevel}.$

\begin{defn}
A map $f_{\sigma}$ corresponding to proper scaling data \break
$\sigma : \mathbb{N} \rightarrow \Delta$ is called infinitely renormalizable if for
$n \geq 1$
\begin{enumerate}
\item [(i)] $[f_{\sigma}(x_{n-1}), 1]$ is the maximal domain containing $1$ on which
$f_{\sigma}^{2^n-1}$ is defined affinely.
\item [(ii)] $f_{\sigma}^{2^{n}-1}\left([f_{\sigma}(x_{n-1}\right), \;1]) = I_0^n.$
\end{enumerate}
\end{defn}
Define $W = \{ f_{\sigma} \ : f_{\sigma} \;\text{is infinitely renormalizable}\}.$
Let $f \in W$ be given by the proper scaling data
$\sigma : \mathbb{N} \rightarrow \Delta$ and define
$$ \hat{I_0^n} = [q_c(x_{n-1}),\;1] = [f(x_{n-1}), \;1].$$
Let
$$h_{\sigma,\; n} :[0,1] \rightarrow [0, 1]$$
be defined by
$$h_{\sigma,\;n}  =  \sigma_0(1) \circ \sigma_0(2) \circ \dots \circ
\sigma_0(n).$$
Furthermore let
$$ \hat{h}_{\sigma,\;n} :  [0, 1] \rightarrow \hat{I_0^n}$$
be the affine orientation preserving homeomorphism.
Then define
$$\mathrm{R}_n f_{\sigma} : h_{\sigma, n}^{-1}(D_{\sigma}) \rightarrow [0, 
\;1]$$
by
$$\mathrm{R}_n f_{\sigma} = \hat{h}^{-1}_{\sigma,\;n} \circ f_{\sigma} 
\circ
h_{\sigma,\; n}.$$

\begin{figure}[ht]
\centering
\psfrag{i0n}[c][c][0.8][0]{${\rm I}_0^n$}
\psfrag{i0p}[c][c][0.8][0]{${\rm \hat{I}}_0^n$}
\psfrag{Hs}[c][c][0.8][0]{$h_{\sigma, n}$}
\psfrag{Hsp}[c][c][0.8][0]{$\hat{h}_{\sigma, n}$}
\psfrag{fs}[c][c][1][0]{$f_{\sigma}$}
\psfrag{Rnf}[c][c][1][0]{${R_{n}f}$}
\includegraphics[scale=1]{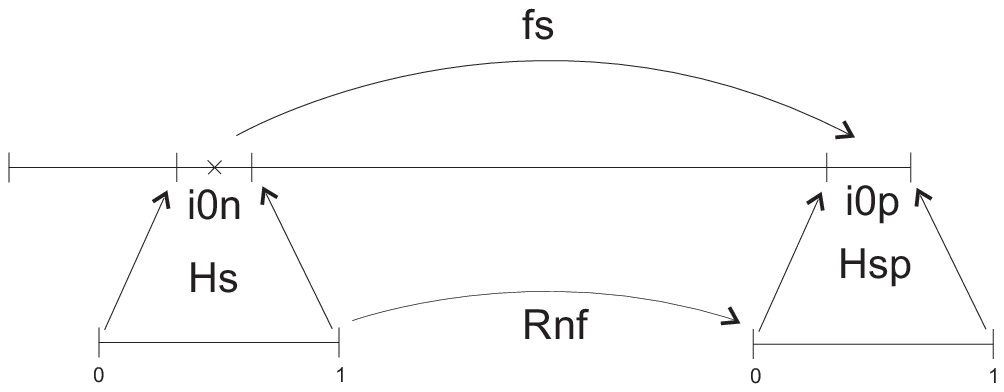}
\caption{}
\label{scaling}
\end{figure}


It is shown in Figure $\ref{scaling}$.
Let $s: \Delta^{\mathbb{N}} \rightarrow \Delta^{\mathbb{N}}$ be the shift
$$s(\sigma)(k) = \sigma(k+1).$$
The construction implies the following result:
\begin{lem} \label{lem1} Let $\sigma: \mathbb{N} \rightarrow \Delta$ be proper
scaling data such that $f_{\sigma}$ is infinitely renormalizable. Then
$$ \mathrm{R}_n f_{\sigma} = f_{s^n(\sigma)}.$$
\end{lem}
Let next $f_{\sigma}$ be infinitely renormalizable, then for $n \geq 0$ we 
have
$$f_{\sigma}^{2^n}: \mathrm{D}_{\sigma} \cap I_0^n \rightarrow I_0^n$$
is well defined.
Define the renormalization $\mathrm{R}: W \rightarrow W$ by
$$\mathrm{R}f_{\sigma} = {h}^{-1}_{\sigma,\;1} \circ f_{\sigma}^2 \circ
h_{\sigma,\; 1}.$$
The map $f_{\sigma}^{2^n-1}: \hat{I}_0^n \rightarrow I_0^n$ is an affine
homeomorphism whenever \break
$f_{\sigma} \in W$. This implies immediately the following
Lemma.
\begin{lem} \label{lem2}
$\mathrm{R}^n f_{\sigma}: \mathrm{D}_{s^n(\sigma)} \rightarrow [0, 1]
\;\; \text{and} \;\;\mathrm{R}^n f_{\sigma} = \mathrm{R}_n f_{\sigma}.$
\end{lem}
\begin{prop}
$W = \{f_{\sigma^{*}}\}$ where $\sigma^{*}$ is 
characterized by $\mathrm{R}f_{\sigma^{*}} = f_{\sigma^{*}}$
\end{prop}
\begin{proof}
Let $\sigma : \mathbb{N} \rightarrow \Delta$ be proper scaling data such that
$f_{\sigma}$ is infinitely renormalizable. Let $c_n$ be the critical point of
$f_{s^n(\sigma)}$. Then
\begin{eqnarray}
q_{c_n}(0) & = & 1-\sigma_1(n) \label{cond1}\\
q_{c_n}(1-\sigma_1(n)) & = & \sigma_0(n) \label{cond2}\\
c_{n+1} & = & \frac{\sigma_0(n)-c_n}{\sigma_0(n)}. \label{cond3}
\end{eqnarray}
We also have the conditions
\begin{eqnarray}
\sigma_0(n), \sigma_1(n) & > & 0 \label{cond4}\\
\sigma_0(n) + \sigma_1(n) & < & 1 \label{cond5} \\
0 < c_n & < & \frac{1}{2} \label{cond6}
\end{eqnarray}
From conditions $(\ref{cond1}), (\ref{cond2}) \;\text{and} \;(\ref{cond3})$ we get
\begin{eqnarray}
\sigma_0(n) & = & \frac{2c_n^2-6c_n^3+5c_n^4-2c_n^5}{(c_n-1)^6} \equiv A_0(c_n)\\
\sigma_1(n) & = & \frac{c_n^2} {(c_n-1)^2} \equiv A_1(c_n)\\
c_{n+1} & =  &
\frac{c_n^6-6c_n^5+17c_n^4-25c_n^3+21c_n^2-8c_n+1}{2c_n^4-5c_n^3+6c_n^2-2c_n}
\equiv  R(c_n)  \label{cnew}
\end{eqnarray}

\begin{figure}[ht]
\centering \psfrag{s0}[c][c][0.8][0]{${A_0(c)}$}
\psfrag{c}[c][c][1][0]{${c}$}
\includegraphics [scale=0.5] {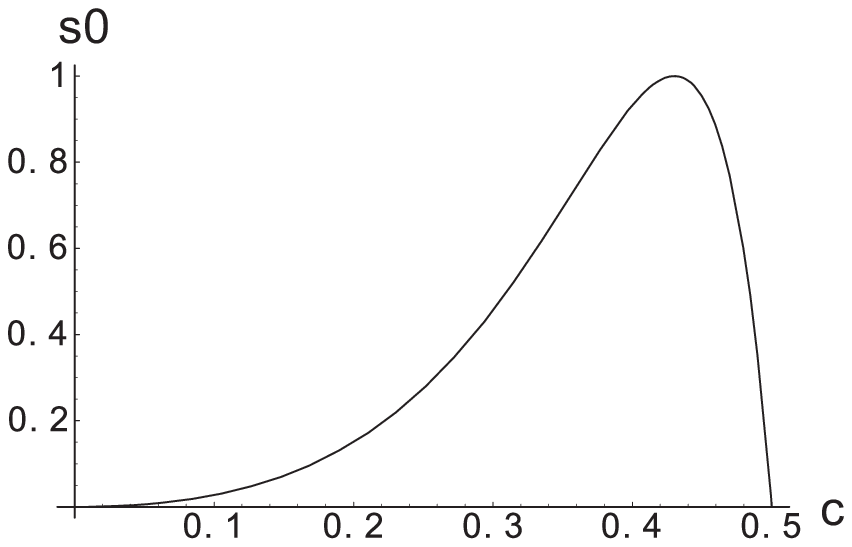} \hspace{2cm}
\psfrag{s1}[c][c][0.8][0]{${A_1(c)}$} \psfrag{C}[c][c][1][0]{${c}$}
\includegraphics [scale=0.5] {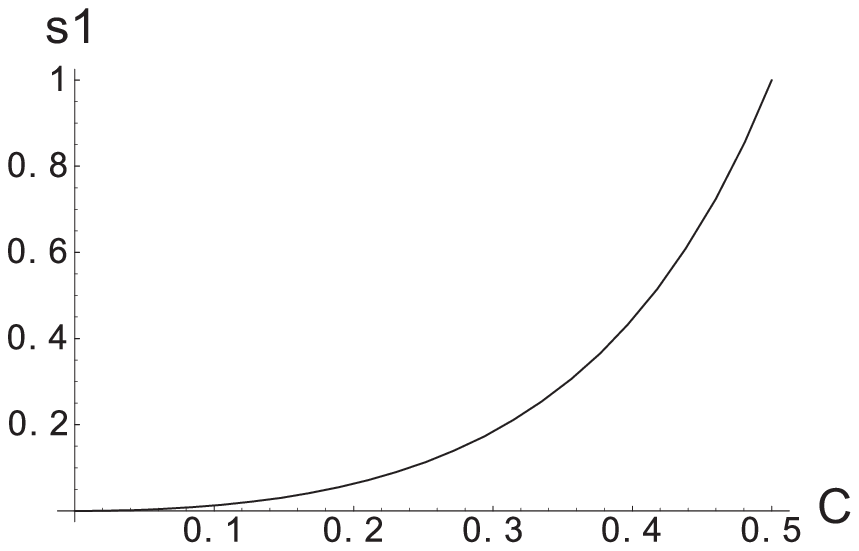}
\end{figure}


\begin{figure}[ht]
\centering \psfrag{SS}[c][c][0.8][0]{${A_0(c)+A_1(c)}$}
\psfrag{c}[c][c][1][0]{${c}$}
\psfrag{C}[c][c][1.25][0]{${C}$}
\includegraphics[scale=0.75]{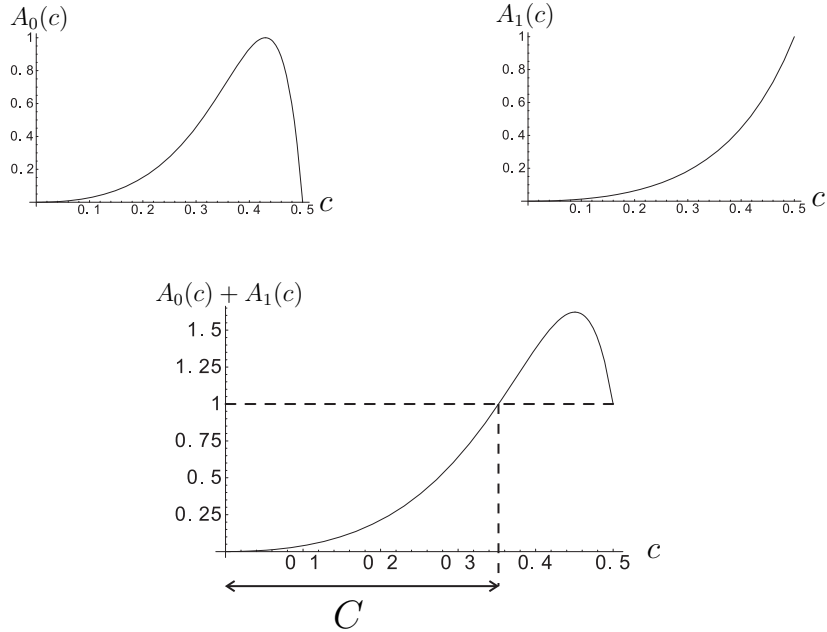}
\caption{The graphs of $A_0, A_1 \;\text{and} \;A_0+A_1$}
\end{figure}


The conditions $(\ref{cond4}), (\ref{cond5}) \;\text{and} \;(\ref{cond6})$ reduces to
$c \in (0, 1/2)$ and \break
$A_0(c)+A_1(c) <1$. In particular this lets the feasible domain be:
\begin{eqnarray*}
C & = & \left \{ c \in (0, \;1/2) \;:\; 0 \leq
\frac{c^2(3-10c+11c^2-6c^3+c^4)}{(c-1)^6} < 1 \right\}\\
& = & [0, \;0.35...]
\end{eqnarray*}

\begin{figure}[ht]
\centering
\psfrag{R}[c][c][1][0]{${R}$}
\psfrag{c}[c][c][1][0]{${c}$}
\psfrag{C}[c][c][1][0]{${C}$}
\psfrag{Cs}[c][c][1][0]{${c^{*}}$}
\includegraphics[scale=0.75]{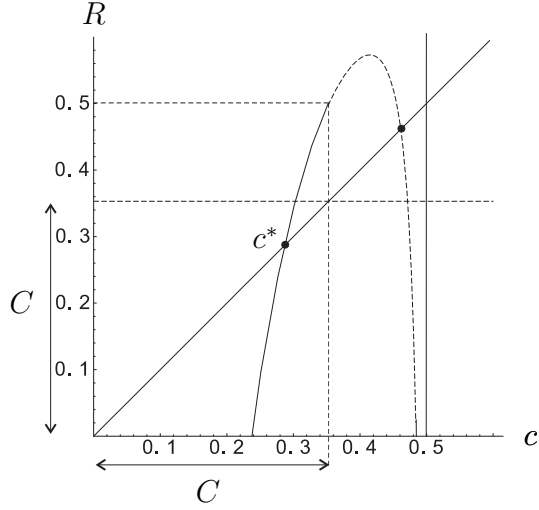}
\caption{$R:C \rightarrow \mathbb{R}$}
\end{figure}


Notice that the map $R: C \rightarrow \mathbb{R}$ is expanding. It 
follows readily that only the
fixed point $c^{*} \in C$ and $R(c^{*}) = c^{*}$ corresponds to an infinitely
renormalizable
$f_{\sigma^{*}}.$ Otherwise speaking, consider the scaling data 
$\sigma^{*}: \mathbb{N}
\rightarrow \Delta$ with
$$\sigma^{*}(n) = \left( q_{c^*}^2(0), \; 1-q_{c^{*}}(0)\right), \; n \geq 1.$$
Then $s(\sigma^{*})= \sigma^{*}$ and
Lemma $\ref{lem1}$ implies
$$\mathrm{R}f_{\sigma^{*}} = f_{\sigma^{*}}.$$
\end{proof}
\begin{rem}
Let $I_0^n = [x_{n-1}, x_n]$ be the interval corresponding to $\sigma^{*}$ then
$$f_{\sigma^{*}}(x_{n-1}) = q_{c^{*}}(x_{n-1}).$$
Hence $f_{\sigma^{*}}$ has a quadratic tip.
\end{rem}

\begin{rem}
The invariant Cantor set of the map $f_{\sigma^{*}}$ is next in complexity to the well known middle
third Cantor set in the following sense:

- like in the middle third Cantor set, on each scale and everywhere the same scaling ratios are used,

- but unlike in the middle third Cantor set, there are now two ratios (a small one and a bigger one) at each scale .

\noindent
This situation of rather extreme tameness of the scaling data is very different from the geometry of the Cantor attractor of the analytic renormalization fixed point in which there are no two places where the same scaling ratios are used at all scales, and where the closure of the set of ratios is itself a Cantor set \cite{BMT}.
\end{rem}

\begin{lem} \label{lip-factor}
Let $f_{*} = f_{\sigma^{*}}$ where $\sigma^{*}: \mathbb{N}
\rightarrow \Delta$ is the scaling data with $\sigma^{*}(n)(\sigma_0^{*}, \sigma_1^{*})$. Then
$$(\sigma_0^{*})^2 = \sigma_1^{*}.$$  \label{sigsq-sig}
\end{lem}
\begin{proof}
Let $\hat{I}_0^n =f_{*}(I_0^n)=[f_{*}(x_{n-1}), \; 1]$ and
$\hat{I}_1^{n+1} =f_{*}(I_1^{n+1})$. Then
$f_{*}^{2^n-1}:\hat{I}_0^n \rightarrow I_0^n$ is affine, monotone
and onto. Further, by construction
$$f^{2^n-1}(\hat{I}_0^{n+1}) = I_1^{n+1}.$$
Hence, $$\frac{|\hat{I}_0^{n+1}|}{|\hat{I}_0^n|} = \sigma_1^{*}.$$
So $|I_0^n| = (\sigma_0^{*})^n$ and $|\hat{I}_0^n| =(\sigma_1^{*})^n$. Now
$f_{\sigma^{*}}$ has
a quadratic tip with
$$f_{\sigma^{*}}(x_n)=q_{c_{*}}(x_n).$$
Hence,
$$\sigma_1^{*} = \frac{|\hat{I}_0^{n+1}|}{|\hat{I}_0^n|} = \left( 
\frac{x_n-c}{x_{n-1}-c} \right)^2 = \left( \frac{|I_0^{n+1}|}{|I_0^n|} 
\right)^2 = \left( \sigma_0^{*} \right)^2.$$ 
This completes the proof.
\end{proof}

\subsection{$C^{1+Lip}$ extension}
In this sub-section we will extend the piece-wise affine map $f_{*}$ to a 
$C^{1+Lip}$ unimodal map.
Let $S:[0, 1]^2 \rightarrow [0, 1]^2$ be the scaling function defined by
\begin{displaymath}
S \left( \begin{array}{cc}
x \\
y
\end{array} \right) 
= \left( \begin{array}{cc}
-\sigma_{0}^{*}x+\sigma_{0}^{*} \\
\sigma_{1}^{*}y+1-\sigma_{1}^{*}
\end{array} \right) 
\equiv  \left( \begin{array}{cc}
S_1(x)\\
S_2(y)
\end{array} \right)
\end{displaymath}
and let $F$ be the graph of  $f_{*} = f_{\sigma^{*}}$,  where  $f_{\sigma^{*}}: D_{\sigma^{*}} \rightarrow [0,1]$, \break
$D_{\sigma^{*}} = \cup_{n \geq 1} I_1^n.$ Then the idea of how to 
construct an
extension $g$ of $f_{*}$ is contained in the following lemma:
\begin{lem} $F \cap \;image(S) = S(F)$. \label{scaling-graph}
\end{lem}
\begin{proof} Let $\hat{h} = \hat{h}_{\sigma^{*}, 1}$ and $h = h_{\sigma^{*}, 1}$.
Let $(x, y) \in \;graph(f_{*}) \;\cap \;image(S)$. Say 
$(x, y) = \left(S_1(x^{\prime}), S_2(y^{\prime}) \right)$ with 
$S_2(y^{\prime}) = f_{*}(S_1(x^{\prime}))$. Since $S_1(x^{\prime}) = 
h(x^{\prime})$ and $S_2(y^{\prime}) = \hat{h}(y^{\prime})$, we can
write
$y^{\prime} = \hat{h}^{-1}\circ f_{*} \circ h(x^{\prime})$. By Lemma $\ref{lem1}$
$$y^{\prime}= R_1f_{*}(x^{\prime}) = f_{*}(x^{\prime}),$$
which gives $(x^{\prime}, y^{\prime}) \in graph (f_{*})$. This in 
turn implies $(x, y) \in S(graph
f_{*})$. By reading the previous argument backward we prove $S(\text{graph}\; f_{*}) \subset F \cap image(S)$.
\end{proof}
\begin{lem} $S(graph \;q_{c^{*}}) \subset \;graph(q_{c^{*}}).$ \label{scaling-qud-tip}
\end{lem}
\begin{proof}
Let $S(graph(q_{c^{*}}))$ be the graph of the function $q$. Since
$S$ is linear and $q_c$ is quadratic we get that $q$ is also a
quadratic function. Then both $q_{c^{*}}(c^{*})=1$ and
$q(c^{*})=1$, because of $S(c^{*},1) =(c^{*},1)$. Furthermore, by
construction
$$S(1, 0)=(0, q_{c^{*}}(0)) = (0,q(0)).$$
Hence  $q_{c^{*}}(0) = q(0)$. Differentiate twice $S_2(y) = q(S_1(x))$ and use $(\sigma_0^{*})^2 = \sigma_1^{*}$ from
Lemma
$\ref{sigsq-sig}$, which proves $q^{''}(c^{*}) = q_{c^{*}}^{''}(c^{*})$. Now we conclude that the quadratic maps
$q$ and $q_{c^*}$ are equal.
\end{proof}

Let $F_0$ be the graph of $f_{*}|_{I_{1}^{1}}$. Then by Lemma $\ref{scaling-graph}$,  $F = \cup_{k \geq 0}
S^{k}(F_0)$. Let $g$ be a $C^{1+Lip}$
extension of $f_{*}$ on $D_{\sigma_{*}} \cup [x_1, 1]$ and $G_0 = graph \;(g|_{[x_1, \;1]}).$ Then
$G = \cup_{k \geq 0} S^k(G_0)$ is the graph of an  extension of $f_*$. We prove that $g$ is $C^{1+Lip}$
and also has a quadratic tip. Let $B^k = S^k([0, 1]^2)$, where
\begin{eqnarray*}
B^k & = & [x_{k-1}, \;x_{k}] \times [\hat{x}_{k-1}, \;1] \qquad  \text{for} \;k=1,3,5,\dots \\
B^k & = & [x_{k}, \;x_{k-1}] \times [\hat{x}_{k-1}, \;1] \qquad \text{for} \;k=2,4,\dots
\end{eqnarray*}
where $\hat{x}_{k-1} = q_c(x_{k-1}) = 1-(\sigma_1^{*})^k$. Let $b_n =\left( x_{n-1}, \hat{x}_{n-1} \right) = S^n(1, 0)$.

\begin{rem}
Notice that the points $b_n$ lie on the graph of $q_{c^*}$. This follows from Lemma
$\ref{scaling-qud-tip}$.
\end{rem}

\begin{figure}[ht]
\centering \psfrag{B0}[c][c][1][0]{${B_0}$}
\psfrag{B1}[c][c][1][0]{${B_1}$}\psfrag{B2}[c][c][1][0]{${B_2}$}
\psfrag{B3}[c][c][1][0]{${B_3}$} \psfrag{B4}[c][c][1][0]{${B_4}$}
\psfrag{b1}[c][c][1][0]{${b_1}$}\psfrag{b2}[c][c][1][0]{${b_2}$}
\psfrag{b3}[c][c][1][0]{${b_3}$} \psfrag{b4}[c][c][1][0]{${b_4}$}
\psfrag{G0}[c][c][1][0]{${G_0}$}\psfrag{G1}[c][c][1][0]{${G_1}$}
\psfrag{X0}[c][c][1][0]{${x_0}$} \psfrag{X1}[c][c][1][0]{${x_1}$}
\psfrag{X2}[c][c][1][0]{${x_2}$} \psfrag{X3}[c][c][1][0]{${x_3}$}
\psfrag{x0c}[c][c][1][0]{${\hat{x}_0}$}
\psfrag{x1c}[c][c][1][0]{${\hat{x}_1}$}
\psfrag{x2c}[c][c][1][0]{${\hat{x}_2}$}
\includegraphics[scale=1]{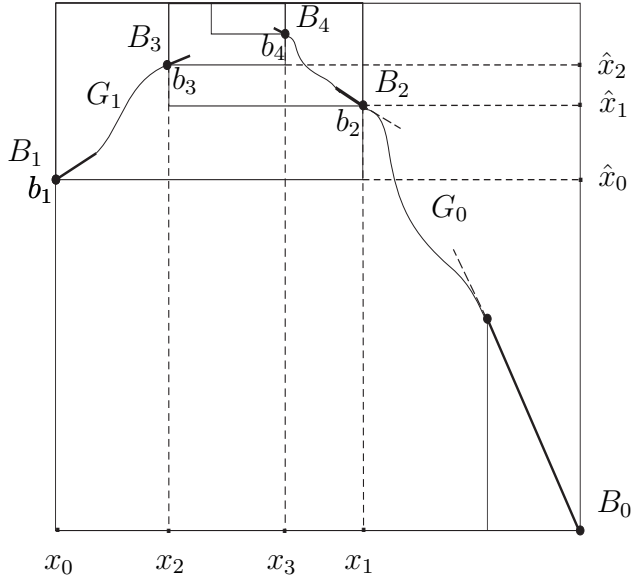}
\caption{extension of $f_{\sigma_{*}}$}
\end{figure}


\begin{lem} G is the graph of a $C^1$ extension of $f_{*}$.
\end{lem}
\begin{proof}
Note that $G_k = S^k(G_0)$ is the graph of a $C^1$ function on
$[x_{k-1}, x_{k+1}]$ for $k$ odd and on $[x_{k+1}, x_{k-1}]$
for $k$ is even. To prove the Lemma we need to show continuous
differentiability at the points $b_n$, where these graphs
intersect. By construction $G_0$ is $C^1$ at $b_2$. Namely,
consider a small interval $(x_1-\delta, x_1+\delta)$. Then on the
interval $(x_1-\delta, x_1)$, the slope is given by an affine
piece of $f_{*}$ and on $(x_1, x_1+\delta)$ the slope is given by
the chosen $C^{1+Lip}$ extension. Let $\Gamma \subset G$ be the
graph over this interval $(x_1-\delta, x_1+\delta)$. Then locally
around $b_n$ the graph $G$ equals $S^{n-1}(\Gamma)$. Hence $G$ is
$C^1$ on $[0, 1] \setminus \left\{ c^{*} \right \}$. From Lemma
$\ref{lip-factor}$, notice that the vertical contraction of $S$ is
stronger than the horizontal contraction. This implies that the
slope of $G_n$ tends to zero. Indeed, $G$ is the graph of a
$C^{1}$ function on $[0, 1]$.
\end{proof}

\begin{prop}
Let $g$ be the function whose graph is $G$ then $g$ is $C^{1+Lip}$ with a quadratic tip.
\end{prop}
\begin{proof}
Since $f_{*}|_{D_{\sigma}}$ has a quadratic tip, the extension $g$ has a quadratic tip. Because $g$ is $C^1$ we only need
to show that $G_n$ is the graph of a $C^{1+Lip}$ function
$$g_n: [x_{n-1}, x_{n+1}] \rightarrow [0, 1]$$ 
with an uniform Lipschitz bound. That is, for $n \geq 1$
$$ Lip(g_{n+1}^{\prime})\leq  Lip(g_n^{\prime}).$$
Assume that $g_n$ is $C^{1+Lip}$ with Lipschitz constant $Lip_n$
for its derivative. We prove that $Lip_{n+1} \leq Lip_n$, and in
particular $Lip_n \leq Lip_0$. For, given $(x, y)$ on the graph of
$g_{n}$ there is $(x^{\prime}, y^{\prime}) = S(x, y)$, on the
graph of $g_{n+1}$. Therefore, we can write
$$g_{n+1}(x^{\prime}) = \sigma_1^{*} \;g_n(x) + 1-\sigma_1^{*}.$$
Since $\displaystyle{x = 1-\frac{x^{\prime}}{\sigma_0^{*}}}$, we have
$$g_{n+1}(x^{\prime}) = \sigma_1^{*} \;g_n \left(
\displaystyle{1-\frac{x^{\prime}}{\sigma_0^{*}}}\right)
+ 1-\sigma_1^{*}.$$
Differentiate,
$$g_{n+1}^{'}(x^{\prime}) = \frac{-\sigma_1^{*}}{\sigma_0^{*}} \; g_n^{'}
\left( \displaystyle{1-\frac{x^{\prime}}{\sigma_0^{*}}}\right).$$
Therefore,
\begin{eqnarray*}
\big| g_{n+1}^{'}(x_1^{\prime}) -g_{n+1}^{'}(x_2^{\prime}) \big| 
& = & \Big| \frac{-\sigma_1^{*}}{\sigma_0^{*}} \Big| \cdot  
\Big| g_n^{'} 
\left(\displaystyle{1-\frac{x_1^{\prime}}{\sigma_0^{*}}}\right)-g_n^{'}
\left(\displaystyle{1-\frac{x_2^{\prime}}{\sigma_0^{*}}}\right) 
\Big| \\
& \leq & \frac{\sigma_1^{*}}{(\sigma_0^{*})^2} \;Lip (g_n^{'}) \; |x_1^{\prime} -x_2^{\prime}|
\end{eqnarray*}
From Lemma $\ref{sigsq-sig}$ we have $\frac{\sigma_1^{*}}{(\sigma_0^{*})^2}=1$. Hence
$$ Lip(g_{n+1}^{'}) \leq Lip(g_n^{'}) \leq Lip(g_1^{'}).$$
which completes the proof.
\end{proof}
\begin{rem}
Notice that if $f_{\sigma}$ is infinitely renormalizable then every extension $g$ is renormalizable in the classical
sense.
\end{rem}
\begin{thm}
There exists an infinitely renormalizable $C^{1+Lip}$ unimodal map
$f$ with a quadratic tip which is not $C^2$ but
$$Rf = f.$$
\end{thm}

\subsection{Entropy of renormalization}
For all $\phi \in C^{1+Lip}$,
$\phi:[x_1,   1] \rightarrow [0, 1]$, which extends $f_{*}$ we constructed $f_\phi \in C^{1+Lip}$
in such a way that
\begin{enumerate}
\item[(i)] $\mathrm{R}f_\phi = f_\phi$
\item[(ii)] $f_{\phi}$ has a quadratic tip.
\end{enumerate}
Now choose two $C^{1+Lip}$ functions which extend $f_{*}$, say
$\phi_0: [x_1, 1] \rightarrow [0, 1]$ and $\phi_1: [x_1,
1] \rightarrow [0, 1]$. For $\omega =(\omega_k)_{k \geq 1} \in \{0,1 \}^{\mathbb{N}}$, define
$$F_n(\omega) =S^n \left(graph \;\phi_{\omega_n}\right)$$
and
$$ F(\omega) = \cup_{k \geq 1} F_k (\omega).$$
Then $F(\omega)$ is the graph of $C^{1+Lip}$
with a quadratic tip $f_\omega$, by an argument similar to what is given 
above. Let now
$$ \tau: \{0,1\}^{\mathbb{N}} \rightarrow \{0,1\}^{\mathbb{N}}$$
be the shift map defined by
$$\tau(\omega)_n = \omega_{n+1},$$
(so that the map $\tau$ acting on the set $\{0, 1\}^{\mathbb{N}}$ is the 
full $2$-shift).
\begin{prop}
For all $\omega \in \{0, 1\}^{\mathbb{N}}$
$$f_{\omega}^2:[0, x_1] \rightarrow [0, x_1]$$
is a unimodal map. In particular $f_{\omega}$ is renormalizable and
$$Rf_{\omega} = f_{\tau(\omega)}.$$
\end{prop}
\begin{proof}
Note that $f_{\omega}:[0, x_1] \rightarrow I_1^1$ is unimodal and onto. Furthermore, $f_{\omega}:I_1^1 \rightarrow [0,
x_1]$ is affine and onto. Hence $f_{\omega}$ is renormalizable. The construction also gives
$$Rf_{\omega} =f_{\tau(\omega)}.$$
\end{proof}
\begin{thm}
Renormalization acting on the space of $C^{1+Lip}$ unimodal maps has positive entropy.
\end{thm}
\begin{proof}
Note that $\omega \rightarrow f_{\omega} \in C^{1+Lip}$ is injective. Hence the domain of $R$ contains a copy of the full
$2$-shift (i.e., contains a subset on which the restriction of $R$ is 
topologically conjugate to the full $2$-shift).
\end{proof}

\begin{rem}
We can also embedded a full $k$-shift in the domain of $R$ by choosing $\phi_0, \phi_1, \dots, \phi_{k-1}$ and repeat the
construction. The entropy of $R$ on $C^{1+Lip}$ is actually unbounded.
\end{rem}

\section{Chaotic scaling data}

In this section we will use a variation on the construction of scaling data as presented in
$\S~\ref{ren-of-c1lip}$  to obtain the following

\begin{thm} \label{chaotic-invar-cantor}
There exists an infinitely renormalizable $C^{1+Lip}$ unimodal map $g$ with quadratic tip such that
$\displaystyle{\left \{c_n \right \}_{n \geq 0}}$, where $c_n$ is the critical point of $R^ng$, is dense in a 
Cantor set. 
\end{thm}
The proof needs some preparation. For $\epsilon > 0$ we will modify the construction as described in
$\S~\ref{ren-of-c1lip}$. This modification is illustrated in Figure $\ref{eps-variation}$. For $c \in (0, 
\frac{1}{2})$
let
\begin{eqnarray*}
\sigma_1(c, \epsilon) & = & 1-q_c(0),\\
\sigma_0(c, \epsilon) & = & \epsilon \;q_c^2(0),
\end{eqnarray*}
where $\epsilon > 0$ and close to $1$. Also let
$$R(c,\epsilon) = \frac{\sigma_0(c,\epsilon)-c}{\sigma_0(c,\epsilon)} = 1-\frac{c}{q_c^2(0)} \cdot \frac{1}{\epsilon}.$$

\begin{figure}[ht]
\centering
\psfrag{qc}[c][c][1][0]{${q_c}$}
\psfrag{c}[c][c][1][0]{${c}$}
\psfrag{so}[c][c][0.75][0]{${\sigma_0(c, \epsilon)}$}
\psfrag{si}[c][c][0.75][0]{${\sigma_1(c, \epsilon)}$}
\psfrag{p1}[c][c][0.75][0]{${q_c^2(0)}$}
\psfrag{p2}[c][c][0.75][0]{${\epsilon \;q_c^2(0)}$}
\psfrag{lc}[c][c][1][0]{$f_{\sigma}$}
\includegraphics[scale=1.15]{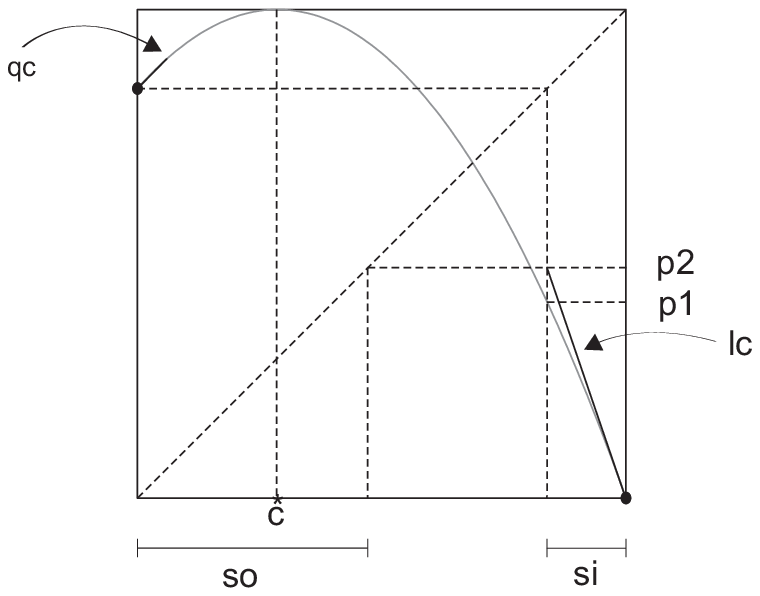}
\caption{} \label{eps-variation}
\end{figure}


In $\S~\ref{ren-of-c1lip}$ we observed that $R(c,1)$ has a unique fixed point
$c^{*} \in (0, \frac{1}{2})$ with feasible
$\sigma_0(c^{*},1)$ and $\sigma_1(c^{*},1)$. This fixed point is expanding. Although we will not use this, a numerical
computation gives
$$\frac{\partial{R}}{\partial{c}}(c^{*},1) > 2.$$
Now choose $\epsilon_0 > \epsilon_1$ close to $1$. Then $R(\cdot, \epsilon_0)$ will have an expanding fixed point
$c^{*}_0$ and $R(\cdot, \epsilon_1)$ a fixed point $c^{*}_1$. In particular, by choosing $ \epsilon_0 > \epsilon_1$ close
enough to $1$ we will get the following horseshoe as shown in Figure
$\ref{horseshoe}$; more precisely there exists an interval
$A_0 = [c^{*}_0, a_0]$ and $A_1 = [a_1, c_1^{*}]$ such that
$$ R_0 : A_0 \rightarrow [c^{*}_0, c_1^{*}] \supset A_0$$
and
$$R_1 : A_1 \rightarrow [c^{*}_0, c_1^{*}] \supset A_1$$
are expanding diffeomorphisms (with derivative larger than $2$, but larger than one would suffice to get a 
horseshoe). Here
$$R_0(c) = R(c, \epsilon_0)$$ and $$R_1(c) = R(c, \epsilon_1).$$

\begin{figure}[ht]
\centering
\psfrag{R0}[c][c][1][0]{${\rm{R}_0}$}
\psfrag{R1}[c][c][1][0]{${\rm{R}_1}$}
\psfrag{co}[c][c][1][0]{${c_0^{*}}$}
\psfrag{ci}[c][c][1][0]{${c_1^{*}}$}
\psfrag{A0}[c][c][1][0]{${\rm{A}_0}$}
\psfrag{A1}[c][c][1][0]{${\rm{A}_1}$}
\includegraphics[scale=0.50]{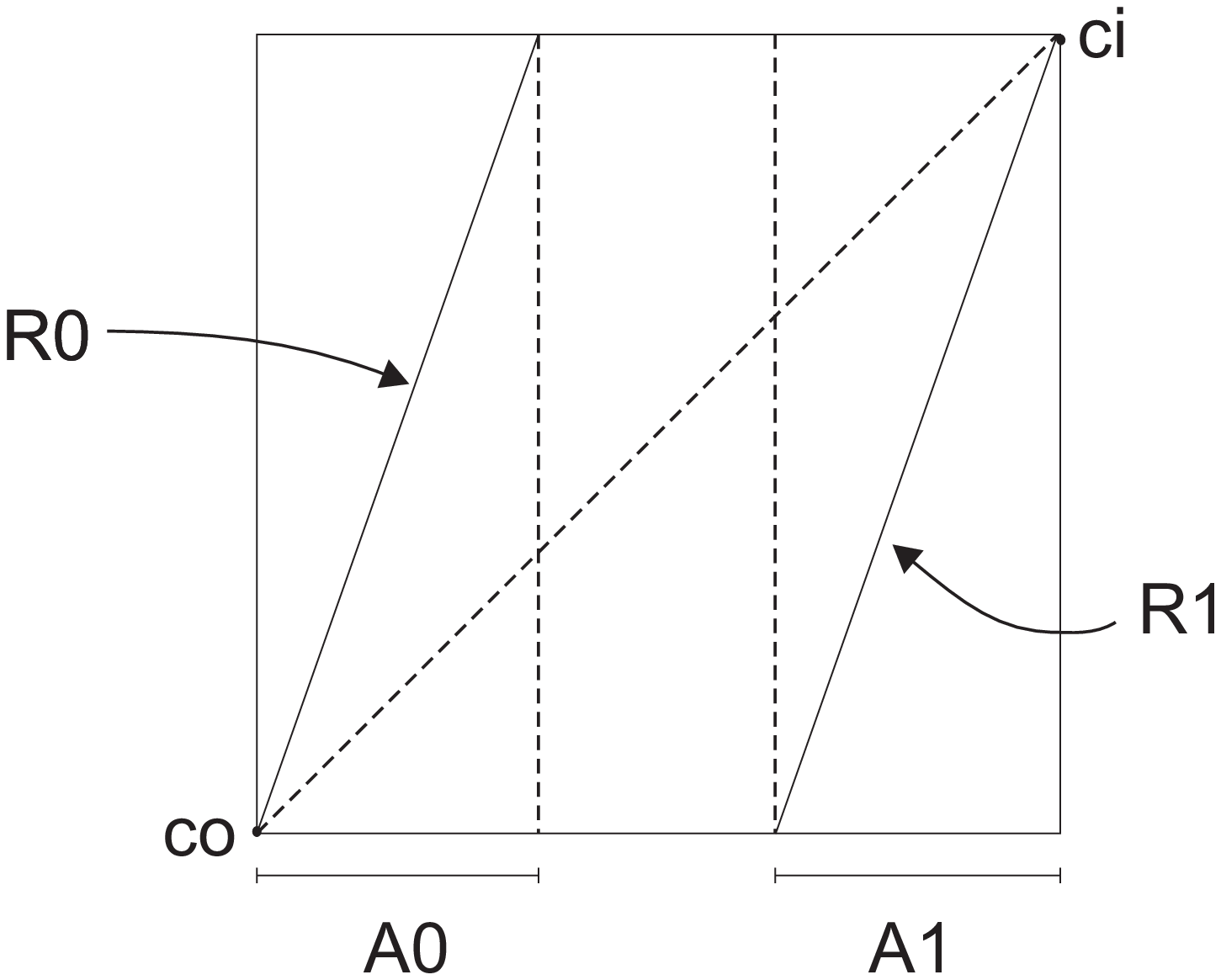}
\caption{} \label{horseshoe}
\end{figure}

Use the following coding for the invariant Cantor set of the horseshoe map
$$c: \left\{0, 1\right \}^{\mathbb{N}} \rightarrow [c^{*}_0, c_1^{*}]$$
with
$$c(\tau \omega) = R \left( c(\omega), \epsilon_{\omega_0} \right)$$
where $\tau:  \left\{0, 1\right \}^{\mathbb{N}} \rightarrow
\left\{0, 1\right \}^{\mathbb{N}}$ is the shift.
Given $\omega \in \left\{0, 1\right \}^{\mathbb{N}}$ define the following scaling
data $\sigma: \mathbb{N} \rightarrow \Delta$.
$$\sigma(n) = \left ( \sigma_0 \left( c(\tau^n \omega), \epsilon_{\omega_n}
\right), \sigma_1 \left( c(\tau^n \omega), \epsilon_{\omega_n} \right) \right).$$
Again, by taking $\epsilon_0, \epsilon_1$, close enough to $1$, we can assume that
$\sigma(n)$ is proper scaling data for any chosen $\omega \in
\left\{0,1\right\}^{\mathbb{N}}.$
As in $\S~\ref{ren-of-c1lip}$ we will define a piece wise affine map
$$f_{\omega}:D_{\omega} = \cup_{n \geq 1} I_1^n \rightarrow [0, 1].$$
The precise definition needs some preparation. Use the notation as 
illustrated in Figure $\ref{chaotic-cantor}$.
For $n \ge 0$ let
$$I_0^n =  [x_n, \;x_{n-1}]$$
where $x_n =  \partial{I_0^n} \setminus \partial{I_0^{n-1}}, \;\; n \geq 1$ and
$$I_1^n  =  [y_n, \;x_{n-2}]$$
where $y_n =  \partial{I_1^n} \setminus \partial{I_0^{n-1}},\;\;  n \geq 1.$

\begin{figure}[ht]\centering
\psfrag{I0n}[c][c][1][0]{${\rm I}_0^n$}
\psfrag{I1n}[c][c][1][0]{${\hat{\rm I}}_0^n$}
\psfrag{Ion1}[c][c][1][0]{${\rm I}_0^{n+1}$}
\psfrag{Iin1}[c][c][1][0]{${\rm I}_1^{n+1}$}
\psfrag{Xn}[c][c][1][0]{${x_{n}}$}
\psfrag{xn-1}[c][c][1][0]{${x_{n-1}}$}
\psfrag{yn}[c][c][1][0]{${y_{n+1}}$}
\psfrag{x-1c}[c][c][1][0]{${\hat{x}_{n-1}}$}
\psfrag{xnc}[c][c][1][0]{${\hat{x}_{n}}$}
\psfrag{ync}[c][c][1][0]{${\hat{y}_{n+1}}$}
\psfrag{c}[c][c][1][0]{${c}$} \psfrag{1}[c][c][1][0]{${1}$}
\psfrag{qc}[c][c][1][0]{${q_c}$}
\psfrag{ln}[c][c][1][0]{${\rm{\hat{I}}_1^{n+1}}$}
\psfrag{rn}[c][c][1][0]{${\rm{\hat{I}}_0^{n+1}}$}
\includegraphics[scale=1.15]{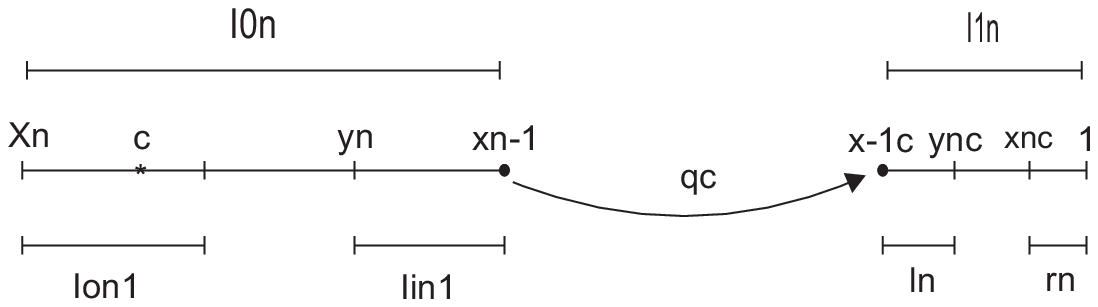}
\caption{}
\label{chaotic-cantor}
\end{figure}

Let $$\hat{I}_0^n = q_c([x_{n-1}, \;1]) = q_c(I_0^n) = [\hat{x}_{n-1}, \;1]$$
where $\hat{x}_{n-1} = q_c(x_{n-1}).$ Finally, let
$\hat{I}_1^{n+1} = [\hat{x}_{n-1}, \; \hat{y}_{n+1}] \subset \hat{I}_0^n$ such that
$$|\hat{I}_1^{n+1}| = \sigma_0(n) \cdot |\hat{I}_0^n|.$$
Now define $f_{\omega} : I_{1}^{n+1} \rightarrow \hat{I}_1^{n+1}$ to be the affine homeomorphism such that


$$f_{\omega} (x_{n-1}) = q_c(x_{n-1}) = \hat{x}_{n-1}.$$
\begin{lem} \label{lip-cond}
There exists $K > 0$ such that
$$\frac{1}{K} \leq \frac{|\hat{I}_0^n|}{|I_0^n|^2} \leq K.$$
\end{lem}
\begin{proof}
Observe, $c(n) = c(\tau^n \omega) \in [c_0^{*}, c_1^{*}]$ which is a small interval around $c^{*}$.
This implies that for some $K > 0$
$$\frac{1}{K} \leq \frac{|c-x_{n-1}|}{|I_0^n|} \leq K.$$
Then
$$\frac{|\hat{I}_0^n|}{|I_0^n|^2} = \frac{|q_c([c,x_{n-1}])|}{|I_0^n|^2} 
= \frac{(c-x_{n-1})^2}{(1-c)^2} \cdot \frac{1}{(I_0^n)^2}$$
which implies the bound.
\end{proof}
Let $S^n_2: [0, 1] \rightarrow \hat{I}_0^n$ be the affine orientation preserving
homeomorphism and $S_1^n: [0, 1] \rightarrow I_0^n$ be the affine homeomorphism
with $S_1^n(1) = x_{n-1}.$ Define
$$S^n: [0, 1]^2 \rightarrow [0, 1]^2$$
by
\begin{displaymath}
S^n \left( \begin{array}{cc}
x \\
y
\end{array} \right) 
= \left( \begin{array}{cc}
S_1^n(x)\\
S_2^n(y)
\end{array} \right).
\end{displaymath}
The image of $S^n$ is $B_n$.

\begin{figure}[ht]
\centering
\psfrag{qc}[c][c][1][0]{${q_c}$}
\psfrag{cn}[c][c][1][0]{${c_n}$}
\psfrag{so}[c][c][0.75][0]{${\sigma_0(n)}$}
\psfrag{si}[c][c][0.75][0]{${\sigma_1(n)}$}
\psfrag{fn}[c][c][0.75][0]{${\rm{F}_n}$}
\psfrag{gn}[c][c][0.75][0]{${\rm{G}_n}$}
\includegraphics[scale=1.15]{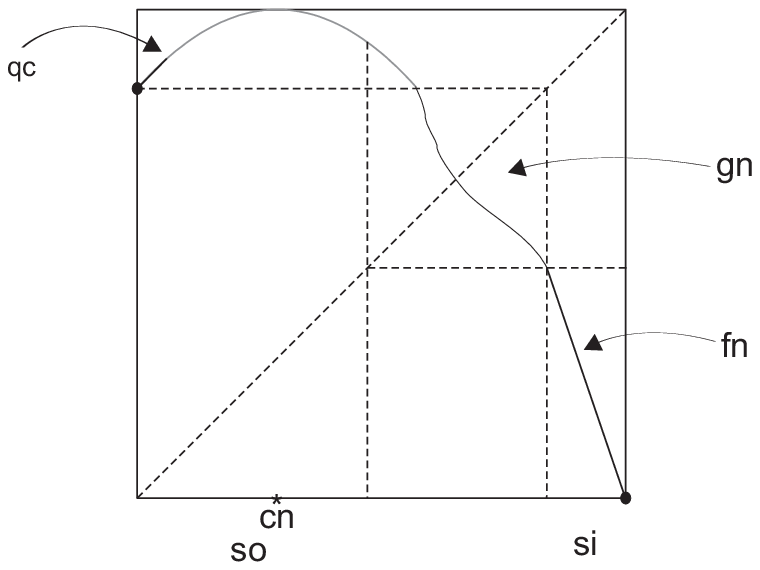}
\caption{}
\end{figure}

Let $F_n = (S^n)^{-1}(graph \;f_{\omega})$. This is the graph of a function $f_n$.
We will extend this function (and its graph) on the gap $[\sigma_0(n), \;
1-\sigma_1(n)].$ Notice, that $\sigma_0(n), 1-\sigma_1(n), Df_n(\sigma_0(n)),\;
\text{and}\; Df_n(1-\sigma_1(n))$
vary within a compact family. This allows us to
choose from a compact family of $C^{1+Lip}$ diffeomorphisms an extension
$$g_n:[\sigma_0(n), 1] \rightarrow [0, f_n(\sigma_0(n))]$$
of the map $f_n$. The Lipschitz constant of $Dg_n$ is bounded by $K_0 > 0$.
Let $G_n$ be the graph of $g_n$ and
$$G = \cup_{n \geq 0 } \;S^n(G_n).$$
Then $G$ is the graph of a unimodal map
$$g:[0, 1] \rightarrow [0,1]$$
which extends $f_{\omega}$. Notice, $g$ is $C^{1}$. It has a quadratic tip because
$f_{\omega}$ has a quadratic tip. Also notice that $S^n(G_n)$
is the graph of a $C^{1+Lip}$ diffeomorphism. The Lipschitz bound $L_n$ of its derivative
satisfies, for a similar reason as in $\S~\ref{ren-of-c1lip}$,
$$L_n \leq \frac{|\hat{I}_0^n|}{|(I_0^n)|^2} \cdot K_0.$$
This is bounded by Lemma $\ref{lip-cond}.$
Thus $g_{\omega}$ is a $C^{1+Lip}$ unimodal map with quadratic tip. The
construction implies that $g$ is infinitely renormalizable and
$$graph \;(R^ng_{\omega}) \supset F_n.$$
One can prove Theorem $\ref{chaotic-invar-cantor}$ by choosing $\omega \in \left\{0, 1\right\}^{\mathbb{N}}$ such that the
orbit under the
shift $\tau$ is dense in the invariant Cantor set of the horseshoe map.
\begin{rem}
Let $\omega = \left\{0,0,\dots \right\}$, then we will get another  renormalization fixed point which is a modification of
the one constructed in $\S~\ref{ren-of-c1lip}.$
\end{rem}

\section{$C^{2+|\cdot|}$ unimodal maps} \label{c2abs}

Let $f:[0, 1] \rightarrow [0, 1]$ be a $C^2$ unimodal map with critical point $c \in (0, 1)$.
Say,  $D^2f(x) = E (1+\varepsilon(x))$, where
$$\varepsilon:[0, 1] \rightarrow \mathbb{R}$$ 
is continuous with $\varepsilon(c)= 0$ and $E = D^2f(c) \neq 0$. Let then
$$\bar {\varepsilon}:[0, 1] \rightarrow \mathbb{R}$$ 
be defined by
$$\bar{\varepsilon}(x) = \frac{1}{x-c} \; \int_c^x \varepsilon(t) dt.$$
Notice, $\bar{\varepsilon}$ is continuous with $\bar {\varepsilon}(c) = 0$. Furthermore, $1+\bar {\varepsilon}(x) \neq
0$ for all $x \in [0, 1]$. Since
$$Df(x) = E (x-c) (1+\bar {\varepsilon}(x))$$
and $Df(x)$ equals zero only when $x = c$.
Let the map $$\delta:[0, 1] \rightarrow \mathbb{R}$$
defined by
$$\delta(x) = \varepsilon(x)-\bar {\varepsilon}(x).$$
Notice that $\delta$ is continuous and $\delta(c)=0$.
Finally, define $$\beta:[0, 1] \rightarrow \mathbb{R}$$ by
$$\beta(x) = \int_c^x \frac{1} {t-c} \;\delta(t) dt.$$

\begin{lem} The function $\beta$ is continuous and $\varepsilon = \delta + \beta$.
\end{lem}
\begin{proof}
The definition of $\delta$ gives $\bar {\varepsilon} = \varepsilon - \delta$, which is differentiable on
$[0, 1]
\setminus \left\{ c \right\}$, and
\begin{eqnarray*}
\varepsilon(x) & = &  \left ( (x-c)(\varepsilon-\delta)(x) \right)^{'}\\
& = & \varepsilon(x) - \delta(x) +(x-c) (\varepsilon-\delta)^{'}(x).
\end{eqnarray*}
Hence, $$\delta(x) = (x-c) (\varepsilon-\delta)^{'}(x).$$
This implies
$$\varepsilon(x) = \delta(x) + \int_c^x \frac{1}{t-c} \; \delta(t) 
dt = \delta(x)+\beta(x).$$
\end{proof}

\begin{defn}
Let $f:[0, 1] \rightarrow [0, 1]$ be unimodal map with critical point $c \in (0, 1)$. We say $f$ is $C^{2+|\cdot|}$ if and only
if
$$ \hat{\beta} : x \longmapsto \int_c^x \;\frac{1}{|t-c|} \;|\delta(t)|dt$$
is continuous.
\end{defn}

\begin{rem}
Every $C^{2+\alpha}$ H\"older unimodal map, $\alpha > 0$, is 
$C^{2+|\cdot|}.$
\end{rem}

\begin{rem}
The very weak condition of local monotonicity of $D^2f$ is sufficient for $f$ to be $C^{2+|\cdot|}$.
\end{rem}

\begin{rem}
$C^{2+|\cdot|}$ unimodal maps are dense in $C^2$.
\end{rem}

\begin{rem}
There exists $C^2$ unimodal maps which are not $C^{2+|\cdot|}$. See also
remark \ref{finalrem}.
\end{rem}

The {\it non-linearity} $\eta_{\phi}: [0, 1] \rightarrow
\mathbb{R}$ of a $C^1$ diffeomorphism \break $\phi:[0, 1]
\rightarrow [0, 1]$ is given by
$$\eta_{\phi}(x) = D \; ln D\phi(x),$$
wherever it is defined.

\begin{prop} \label{non-lin-bound}
Let $f$ be a $C^{2+|\cdot|}$ unimodal map with critical point $c \in (0, 1)$. There exist
diffeomorphisms
$$\phi_{\pm}:[0, 1] \rightarrow [0, 1]$$
such that
\begin{displaymath}
f(x) = \left\{ \begin{array}{ll}
\phi_{+} \left( q_c(x) \right) & x \in [c, \;1]\\
\phi_{-} \left( q_c(x) \right) & x \in [0, \;c]
\end{array} \right.
\end{displaymath}
with
$$\eta_{\phi_{\pm}} \in L^1([0, 1]).$$
\end{prop}

\begin{proof}
It is plain that there exists a $C^1$ diffeomorphism
$$\phi_{+}: [0, \;1] \rightarrow [0, \;1]$$
such that for $x \in [c, 1]$
$$f(x) = \phi_{+} \left( q_c(x) \right).$$
We will analyze the nonlinearity of $\phi_{+}$. Observe that:
$$Df(x)  = -2\; \frac{(x-c)}{(1-c)^2} \;\cdot\;D\phi_{+} \left( q_c(x) 
\right)$$ 
and
\begin{eqnarray}
D^2 f(x) & = & 4\; \frac{(x-c)^2}{(1-c)^4} \;\cdot\;D^2 \phi_{+} \left( q_c(x)
\right)  -2\; \frac{1}{(1-c)^2}\;\cdot
D\phi_{+} \left( q_c(x) \right) \nonumber \\
& = & E \;(1+\varepsilon(x)). \label{non-lin-bd-eqn2}
\end{eqnarray}
As we have seen before, we also have
$$Df(x) = E \;(x-c) \cdot \left( 1 + \bar{\varepsilon}(x) \right).$$
This implies that
\begin{eqnarray} \label{non-lin-bd-eqn}
\eta_{\phi_{+}} \left( q_c(x) \right) = \frac{-(1-c)^2}{2}
\;\cdot \;\frac{\varepsilon(x)-\bar{\varepsilon}(x)}{1+\bar{\varepsilon}(x)}
\;\cdot \frac{1}{(x-c)^2}.
\end{eqnarray}
Therefore, by performing the substitution $u = q_c(x)$, we get:
\begin{eqnarray}\label{inteta}
\int_0^1 |\eta_{\phi}(u)| \;du & = & \int_1^c -2\; |\eta_{\phi_{+}} \left( q_c(x)
\right)| \; \frac{x-c}{(1-c)^2} \;dx \\
& = & \int_c^1 \;
\frac{|\varepsilon(x)-\bar{\varepsilon}(x)|}{1+\bar{\varepsilon}(x)}\; \frac{1}{x-c}\;dx\\
& \leq & \frac{1}{\min\;(1+\bar{\varepsilon})}\;\int_c^1 \frac{|\delta(x)|}{|x-c|} \; dx < \; \infty
\end{eqnarray}
We have proved $\eta_{\phi_{+}} \in L^1([0, 1])$. Similarly one can prove the existence
of a $C^1$ diffeomorphism
$$\phi_{-}: [0, \;1] \rightarrow [0, \;1]$$
such that for $x \in [0, c]$
$$f(x) = \phi_{-}(q_c(x))$$
and
$$\eta_{\phi_{-}} \in L^1([0, 1]).$$
\end{proof}

\section{Distortion of cross ratios}

\begin{defn} Let $J \subset T \subset [0, 1]$ be open and bounded
intervals such that $T \setminus J$ consists of two components $L$ and $R$.
Define the cross ratios of these intervals as
$$D(T, J) = \frac{|J||T|}{|L| |R|}.$$
If $f$ is continuous and monotone on $T$ then define the cross ratio 
distortion of $f$ as
$$ B(f,T,J)= \frac{D(f(T),f(J))}{D(T,J)}.$$
If $f^n|_T$ is monotone and continuous then
\begin{displaymath}
B(f^n, T, J) = \prod_{i=0}^{n-1} B \left(f, f^i(T),f^i(J) \right).
\end{displaymath}
\end{defn}

\begin{defn}
Let $f:[0, 1] \rightarrow [0, 1]$ be a unimodal map and $T \subset [0, 1]$. We say that
$$\left\{ f^i(T) : 0 \leq i \leq n  \right \}$$
has intersection multiplicity $m \in \mathbb{N}$ if and only if for every $x \in [0, 1]$
$$\# \left \{ i \leq n \;\vert \;x \in f^i(T) \right \} \leq m$$
and $m$ is minimal with this property.
\end{defn}

\begin{thm} \label{cr-exp}
Let $f:[0, 1] \rightarrow [0, 1]$ be a $C^{2+|\cdot|}$ unimodal map with critical point
$c \in (0, 1).$ Then there exists $K > 0$,
such that the following holds. If $T$ is an interval such that $f^n|_{T}$ is a diffeomorphism then for any interval $J
\subset T$ with
$cl(J) \subset int(T)$ we have,
$$B(f^n,T,J) \geq \exp \;\{-K \cdot m\}$$
where $m$ is the intersection multiplicity of $\left\{ f^i(T) : 0 \leq i \leq n \right \}.$
\end{thm}
\begin{proof}
Observe that $q_c$ expands cross-ratios. Then Proposition $\ref{non-lin-bound}$
implies
$$B \left(f,f^i(T),f^i(J)\right) >  \frac{D\phi_i(j_i)
\cdot D\phi_i(t_i)}{D\phi_i(l_i) \cdot D\phi_i(r_i)}$$
where $\phi_i = \phi_{+} \;\text{or}\;\phi_{-}$ depending whether
$f^i(T) \subset
[c, \;1] \;\text{or}\; [0,\;c]$
and
\begin{eqnarray*}
j_i & \in & q_c\left(f^i(J)\right), \\
t_i & \in & q_c\left(f^i(T)\right), \\
l_i & \in & q_c\left(f^i(L)\right), \\
r_i & \in & q_c\left(f^i(R)\right).
\end{eqnarray*}
Thus
$$
\begin{aligned}
& ln \;B(f^n,T,J)  = \sum_{i=0}^{n-1} \;ln \;B \left(f,f^i(T),f^i(J)\right)  \geq  \\
&  \sum_{i=0}^{n-1}   \left( ln \;D\phi_i(j_i)-ln \;D\phi_i(l_i) \right) +\left( ln
\;D\phi_i(t_i)-ln \;D\phi_i(r_i) \right)  \geq  \\
- & \sum_{i=0}^{n-1}  |\eta_{\phi_i} (\xi^1_i)| \; |j_i-l_i| + |\eta_{\phi_i}
(\xi^2_i)| \; |t_i-r_i|  \geq  \\
- & 2 \;m  \left(\int|\eta_{\phi_{+}}| + \int|\eta_{\phi_{-}}| \right) = -K \cdot m.
\end{aligned}
$$
Therefore $$B(f^n,T,J) \geq \exp \;\{-K \cdot m\}.$$
\end{proof}
The previous Theorem allows us to apply the Koebe Lemma. See \cite{MS} for a proof.

\begin{lem} (Koebe Lemma) \label{koebe}
For each $K_1 > 0$, $0 < \tau < 1/4$, there exists $K < \infty$
with the following property:\\
Let $g:T \rightarrow g(T) \subset [0,
1]$ be a $C^1$ diffeomorphism on some interval $T$. Assume that
for any intervals $J^*$ and $T^{*}$ with $J^{*} \subset T^{*}
\subset T$ one has
$$B(g, T^{*}, J^{*}) \geq K_1 > 0,$$
for an interval $M \subset T$ such that $cl(M) \subset int(T)$. Let $L, R$
be the components of $T \setminus M$. Then, if:
$$ \frac{|g(L)|}{|g(M)|} \geq \tau \;\; and \;\; \frac{|g(R)|}{|g(M)|} \geq
\tau$$
we have:
$$\forall x, y \in M, \qquad    \frac{1}{K} \leq 
\frac{|g^{'}(x)|}{|g^{'}(y)|} \leq K.$$
\end{lem}
\begin{rem}
The conclusion of the Koebe-Lemma is summarized by saying that $g|_{M}$ has bounded
distortion.
\end{rem}

\section{A priori bounds} \label{ap bounds}

Let $f$ be an infinitely renormalizable $C^{2+|\cdot|}$ unimodal map with quadratic tip at $c \in
(0, 1)$. Let $I_0^n =[f^{2^n}(c), f^{2^{n+1}}(c)]$
be the central interval whose first return map corresponds to the $n^{th}$-renormalization. Here, we study the geometry
of the cycle consisting of the intervals
$$I_j^n =f^j(I_0^n), \;\; j=0,1,\dots,2^n-1.$$
Notice that $$I_j^{n+1}, I_{j+2^n}^{n+1} \subset I_j^n, \;\; j=0,1,\dots,2^n-1.$$
Let $I_l^n$ and $I_r^n$ be the direct neighbors of $I_j^n$ for $3 \leq j \leq 2^n.$

\begin{lem} \label{int-exp} For each $1 \leq i < j$, There exists an interval $T$
which contains $I_{i}^{n}$, such that $f^{j-i}:T \rightarrow [I_{l}^{n},  I_{r}^{n}]$
is monotone and onto.
\end{lem}
\begin{proof}
Let $T \subset [0, 1]$ be the maximal interval which contains $I_{i}^n$ such
that $f^{j-i}|_{T}$ is
monotone. Such interval exists because of monotonicity of $f^{j-i}|_{I_i^n}$. The
boundary points of $T$ are $a, b \in [0, 1]$. Suppose $f^{j-i}(b)$ is to the right
of $I_j^n$. The maximality of $T$ ensures
the existence of $k$, $k < j-i$ such that $f^k(b)=c$. Because $i+k < j \leq 2^n$,
we have $c \notin I_{i+k}^n$ and so $f^{k+1}(T) \supset I_1^n$.
Moreover,
$f^{j-i-(k+1)}|_{f^{k+1}(T)}$ is monotone. Hence $f^{j-i-(k+1)}|_{I_1^n}$ is
monotone. So $1+j-i-(k+1) \leq 2^n.$ This implies that
$f^{j-i}(T)$ contains $I_{1+j-i-(k+1)}^n$. In particular $f^{j-i}(T)$ contains
$I_r^n$. Similarly we can prove $f^{j-i}(T)$ contains $I_l^n$.
\end{proof}

\begin{lem} (Intersection multiplicity) \label{int-mul}
Let $f^{j-i}:T \rightarrow [I_l^n, I_r^n]$ be monotone and onto with $T \supset I_i^n$.
Then for all $x \in [0, 1]$
$$ \# \{k < j-i \;\vert\;f^k(T) \ni x\} \leq 7.$$
\end{lem}

\begin{proof}
Without loss of generality we may restrict ourselves to estimate the intersection
multiplicity at a point $x \in U$, where
$$U = [I_l^n, I_r^n] = [u_l, u_r].$$
Let $c_l \in  I_l^n$ such that $f^{2^n-l}(c_l) = c$ and
$$C_l = [u_l, c_l] \subset I_l^n.$$
Similarly, define
$$C_r = [c_r, u_r] \subset I_r^n.$$
Let $T_k = f^k(T), \qquad k=0,1,....j-i.$\\
{\it Claim:
If $i+k \notin \left \{l,j,r \right \}$ and $T_k \cap U \neq \emptyset$ then
\begin{itemize}
\item [(i)] $I_{i+k}^n \cap U = \emptyset$
\item [(ii)]$ U \cap T_k = I_l^n$ or $C_l$ or $I_r^n$ or $C_r$.
\end{itemize} }
Let $T \setminus I_i^n = L \cup R$ and then we may assume $U \cap
T_k = U \cap L_k$ where \break $L_k = f^k(L)$. This holds because
$I_{i+k}^n \cap U = \emptyset$. Consider the situation where
$$I_r^n \cap L_k \neq \emptyset.$$ The other possibilities can be
treated similarly. Notice that $I_r^n$ cannot be strictly
contained
 in $L_k$. Otherwise there would be a third ``neighbor" of $I_j^n$ in $U.$
Let $a = \partial{L} \cap \partial{T}.$
Notice that $$f^k(a) \in \partial{L_k} \cap I_r^n.$$
Furthermore,
$$f^{j-k}(f^k(a))  \in \partial{U}.$$
This means $f^{j-k}(f^k(a))$ is a point in the orbit of $c$. This holds because
all boundary points of the interval $I_j^n$ are in the orbit of $c$. Hence,
$f^k(a)$ is a point in the orbit of $c$ or $f^k(a)$ is a preimage of $c$. The
first possibility implies $f^k(a) \in \partial{I_r^n}.$ This implies
$$U \cap T_k = U \cap L_k = I_r^n.$$
The second possibility implies $f^k(a) = c_r$ which means
$$U \cap T_k = U \cap L_k = C_r.$$
This finishes the proof of claim. This claim gives $7$ as bound for the intersection
multiplicity.
\end{proof}

\begin{prop} \label{bd}
For $j < 2^n, \;f^{2^{n}-j} : I_j^n \rightarrow I_0^n$ has uniform bounded distortion.
\end{prop}
\begin{proof}
Step$1:$ Choose $j_0 < 2^n$, such that for all $j \leq 2^n,$ we have
\break $|I_{j_0}^n|
\leq |I_{j}^n|$. By Lemma $\ref{int-exp}$ there
exists an interval neighborhood \break
$T_n = L_n^0 \cup I_1^n \cup R_n^0$ such that
$f^{j-1}: T_n \rightarrow [I_l^n, I_r^n] \supset I_{j_0}^n$ is monotone and onto.
Lemma $\ref{int-mul}$ together with Theorem $\ref{cr-exp}$ allow us to apply
the Koebe Lemma $\ref{koebe}$. So, there exists
$\tau_0 > 0 $ such that
$$|L_n^0|, |R_n^0| \geq {\tau_0} \;|I_1^n|.$$
Let $U_n = I_0^n$, $V_n = f^{-1} \left(L_n^0 \cup I_1^n \cup R_n^0 \right)$
and let $L_n^1, R_n^1$ be the components of
$V_n \setminus U_n$. From Proposition $\ref{non-lin-bound}$ we get $\tau_1 > 0$
such that
$$|L_n^1|, |R_n^1| \geq \tau_1 \; |U_n|.$$
Step$2:$ Suppose $W_n = [I_{l_n}^n, I_{r_n}^n]$, where $I_{l_n}^n, I_{r_n}^n$ are the direct neighbors of $U_n$. We
claim that $V_n \subset W_n$. Suppose it is not. Then, say
$I_{r_n}^n \subset int (V_n)$ implies that $f(I_{r_n}^n) \subset int (L_n^1)$. So, $f^{j_0-1}|_{f(I_{r_n}^n)}$ is
monotone, implies that
$r_n + j_0 \leq 2^n$ and
$f^{j_0}(I_{r_n}^n) \subset int ([I_l^n, I_r^n])$.
This contradiction concludes that $V_n \subset W_n$.

\noindent
Step$3:$ Let $L_n, R_n$ be the components of $W_n \setminus U_n$. Then
$$|L_n|, |R_n| \geq \tau_1 \;|U_n|.$$
Step$4:$ For all $j < 2^n$, there exists an interval neighborhood $T_j$ which contains
$I_j^n$ such that $f^{2^n-j}: T_j \rightarrow
W_n$ is monotone and
onto. Now Proposition $\ref{bd}$ follows from the
Lemma $\ref{int-mul}$ together with Theorem $\ref{cr-exp}$ and the Koebe Lemma $\ref{koebe}$.
\end{proof}

\begin{cor} \label{derivative-bound}
There exists a constant $K$ such that
$$\big| Df^{2^n}|_{I_0^n} \big| \leq K.$$
\end{cor}

\begin{proof} Let $x \in I_1^n.$ Then from Proposition $\ref{bd}$ we get $K_1 > 0$
such that for some $x_0 \in I_1^n$
\begin{eqnarray*}
|Df^{2^n-1}(x)| & = & \frac{|I_0^n|} {|I_1^n|} \cdot \bigg\{ \frac{Df^{2^n-1}(x)}
{Df^{2^n-1}(x_0)} \bigg\} \\
& \leq & \frac{|I_0^n|} {|I_1^n|} \cdot K_1.
\end{eqnarray*}
Proposition $\ref{non-lin-bound}$ implies that there exists $K_2 > 0$ such that for
$x \in I_0^n$
$$|Df(x)| \leq K_2 \cdot |x-c|$$
and
$$|I_1^n| \geq \frac{1}{K_2} \cdot |I_0^n|^2.$$
Now for $x \in I_0^n$
\begin{eqnarray*}
|Df^{2^n}(x)| & \leq & K_2 \cdot |x-c| \cdot \frac{|I_0^n|} {|I_1^n|} \cdot K_1 \\
& \leq & K_2 \cdot K_1 \cdot \frac{|I_0^n|^2}{|I_1^n|} \;\leq
K_2^2 \cdot K_1 = K
\end{eqnarray*}
Therefore, we conclude that
$\big|Df^{2^n}|_{I_0^n} \big| \leq K.$
\end{proof}

\begin{defn} (A priori bounds) \label{def-ap}
Let $f$ be infinitely renormalizable. We say $f$ has a priori bounds if there
exists $\tau > 0$ such that for all $n \geq 1$ and $j \leq 2^n$ we have
\begin{eqnarray}
\tau & < & \frac{|I_j^{n+1}|}{|I_j^n|}, \;\;  \frac{|I_{j+2^n}^{n+1}|} {|I_j^n|} \label{ratio} \\
\tau & < & \frac{|I_j^n \setminus \left( I_j^{n+1} \cup I_{j+2^n}^{n+1}  \right)|} {|I_j^n|}
\label{gaps}
\end{eqnarray}
where, $I_j^{n+1}, I_{j+2^n}^{n+1}$ are the intervals of next generation
contained in $I_j^n$.
\end{defn}

\begin{prop} \label{thm-ap}
Every infinitely renormalizable $C^{2+|\cdot|}$ map has a priori bounds.
\end{prop}
\begin{proof}
Step$1.$ There exists $\tau_1 > 0$ such that
$\displaystyle{\frac{|I_0^{n+1}|}{|I_0^n|}} > \tau_1$.\\
Let $I_0^n = [a_n, a_{n-1}]$ be the central interval, and so $a_n = f^{2^n}(c)$.
A similar argument as in the proof of Corollary $\ref{derivative-bound}$ gives $K_1
> 0$ such that
$$|f^{2^n}([a_n, c])| \leq  \left( \frac{|a_n-c|} {|I_0^n|} \right)^2 \cdot |I_0^n| \cdot K_1.$$
Notice that
$$f^{2^n}([a_n, c]) = I_{2^n}^{n+1}.$$
Thus
$$|I_{2^n}^{n+1}| \leq  \frac{|a_n-c|^2} {|I_0^n|} \cdot K_1.$$
Note  $$f^{2^n}(I_{2^n}^{n+1}) = I_0^{n+1} \supset [a_n, c].$$
Therefore, by Corollary $\ref{derivative-bound}$
$$ |a_n -c| \leq |f^{2^n}(I_{2^n}^{n+1})| \leq K \cdot |I_{2^n}^{n+1}| \leq K \cdot
\frac{|a_n-c|^2} {|I_0^n|} \cdot K_1.$$
This implies $$|a_n -c | \geq \frac{1}{K} \cdot |I_0^n|.$$
Which proves $\displaystyle{\frac{|I_0^{n+1}|}{|I_0^n|} > \tau_1}.$\\
Step$2.$ There exists $\tau_2 >0$ such that
$\displaystyle{\frac{|I_{2^n}^{n+1}|} {|I_0^n|}} \geq \tau_2$.\\
From above we get
$$ \tau_1 |I_0^n| \leq |I_0^{n+1}| = |f^{2^n}(I_{2^n}^{n+1})| \leq K
\cdot |I_{2^n}^{n+1}|$$
This proves $$\displaystyle{\frac{|I_{2^n}^{n+1}|}{|I_0^n|} \geq \tau_2}.$$
Step$3.$ There exists $\tau_3 >0$ such that the following holds.
$$\frac{|I_j^{n+1}|}{|I_j^n|}, \;\;  \frac{|I_{j+2^n}^{n+1}|} {|I_j^n|} \geq \tau_3.$$
Because
$$f^{2^n-j}(I_j^{n+1}) =I_0^{n+1}, \;\; f^{2^n-j}(I_j^n) = I_0^n$$
and from Proposition $\ref{bd}$ we get a $K > 0$ such that
$$ \frac{|I_j^{n+1}|}{|I_j^n|} \geq \frac{1}{K} \cdot \frac{|I_0^{n+1}|} {|I_0^n|} \geq \frac{\tau_1}{K}.$$
Hence, $\displaystyle{\frac{|I_j^{n+1}|}{|I_j^n|} \geq \tau_3}$.
Similarly we prove
$\displaystyle{\frac{|I_{j+2^n}^{n+1}|}{|I_j^n|} \geq \tau_3}.$ Which completes the
proof of $(\ref{ratio}).$\\
Step$4.$ To complete the proof of the Proposition, it remains to show that
the gap between the intervals
$I_0^{n+1}, I_{2^n}^{n+1}$ and as well as $I_j^{n+1}, I_{j+2^n}^{n+1}$ are not too
small. Let $$G_n = I_0^n \setminus \left(I_0^{n+1} \cup I_{2^n}^{n+1} \right).$$
We claim that there exists $\tau_4 >0$ such that
$$\displaystyle{\frac{|G_n|}{|I_0^n|} \geq \tau_4}.$$
Let $H_n$ be the image of $G_n$ under $f^{2^n}$. Then  $H_n = f^{2^n}(G_n) \supset
I_{3 \cdot 2^n}^{n+2}$.
The claim follows by using Corollary $\ref{derivative-bound}$ and the bounds we
have so far. Namely,
$$K \cdot |G_n| \geq |H_n| \geq  |I_{3 \cdot 2^n}^{n+2}|  \geq \tau_3 \cdot |I_{2^n}^{n+1}| \geq \tau_3 \cdot \tau_2
\cdot |I_0^n|.$$
This implies
$$|G_n| \geq \tau_4 \cdot |I_0^n|.$$
Step$5.$ Let $G_j^n = I_j^n \setminus \left( I_{j}^{n+1} \cup I_{j+2^n}^{n+1} \right)$,
then there exists $\tau_5 >0$ such that
$$\displaystyle{\frac{|G_j^n|}{|I_j^n|} \geq  \tau_5}.$$
We have $f^{2^n-j}(G_j^n) = G_n$
and $f^{2^n-j}(I_j^n)=I_0^n$. Since $f^{2^n-j}$ has bounded distortion, we
immediately get a constant $K > 0$ such that
$$\frac{|G_j^n|}{|I_j^n|} \geq \frac{1}{K} \cdot \frac{|G_n|}{|I_0^n|} \geq \frac{\tau_4}{K}.$$
This implies $$|G_j^n| \geq \tau_5 \cdot |I_j^n|.$$
This completes the proof of $(\ref{gaps})$.
\end{proof}

\section{Approximation of $f|_{I_j^n}$ by a quadratic map} \label{approx qudratic}

Let $\phi : [0, 1] \rightarrow [0, 1]$ be an orientation preserving $C^2$
diffeomorphism with non-linearity $\eta_{\phi} : [0, 1] \rightarrow \mathbb{R}.$
The norm we consider is
$$|\phi| = |\eta_{\phi}|_{0}.$$
Let $[a, b] \subset [0, 1]$ and $f:[a, b] \rightarrow f([a, b])$
be a diffeomorphism. Let
$$1_{[a\;b]}:[0, 1] \rightarrow  [a, b]$$
and
$$1_{f([a, b])}:[0, 1] \rightarrow f([a, b])$$
be the affine homeomorphisms with $1_{[a, b]}(0) = a \; \text{and} \; 1_{f([a, b])}(0) = f(a).$
The rescaling $f_{[a, b]}:[0, 1] \rightarrow [0, 1]$ is the diffeomorphism
$$f_{[a, b]} = \left( 1_{f([a, b])} \right)^{-1} \circ f \circ 1_{[a, b]}.$$
We say that $0 \in [0, 1]$ corresponds to $a \in [a, b]$.

\begin{prop} \label{approx-qudratic-prop}
Let $f$ be an infinitely renormalizable $C^{2+|\cdot|}$ map with critical point $c \in (0, 1)$. For $n \geq 1$
and $1 \leq j < 2^n$ we have
$$f_{I_j^n} = \phi_{j}^n \circ q_j^n$$
where
$$q_j^n = (q_c)_{I_j^n}: [0, 1] \rightarrow [0, 1]$$
such that $0$ corresponds to $f^j(c) \in I_j^n$ and
$\phi_j^n : [0, 1] \rightarrow [0, 1]$ a $C^2$ diffeomorphism.
Moreover
$$\lim_{n \to \infty} \sum_{j=1}^{2^n-1} |\phi_j^n|= 0$$
\begin{proof}
If $I_j^n \subset [c, 1]$ then use Proposition \ref{non-lin-bound} and 
define
$$\phi_j^n = (\phi_{+})_{q_c(I_j^n)}:[0, 1] \rightarrow [0, 1]$$
such that $0 \in [0, 1]$ corresponds to $q_c \left(f^j(c) \right) \in q_c(I_j^n)$.
In case $I_j^n \in [0, c]$ then let
$$\phi_j^n = (\phi_{-})_{q_c(I_j^n)}:[0, 1] \rightarrow [0, 1]$$
where again $0 \in [0,1]$ corresponds to $q_c \left(f^j(c) \right) \in q_c(I_j^n)$.
Let $\eta_j^n$ be the non-linearity of $\phi_j^n$. Then the chain rule for
non-linearities \cite{M} gives
$$|\eta_j^n(x)| = |q_c(I_j^n)| \cdot |\eta_{\phi_{\pm}}(1_j^n(x))|$$
where $1_j^n:[0,1] \rightarrow q_c(I_j^n)$ is the affine homeomorphism such that
$1_j^n(0) = q_c(f^j(c)).$ Now use $(\ref{non-lin-bd-eqn})$ to get
\begin{eqnarray*}
|\eta_j^n|_0 & \leq & |q_c(I_j^n)| \cdot \frac{(1-c)^2}{2} \cdot
\frac{1}{\min_{x \in I_j^n} \left( 1+\bar{\epsilon}(x) \right)} \cdot
\sup_{x \in I_j^n} \frac{|\delta(x)|}{(x-c)^2}\\
& \leq & \frac{1}{\min_{x \in [0, 1]} \left(1+\bar{\epsilon}(x)\right)} \cdot
|\zeta_j^n-c| \cdot |I_j^n| \cdot \sup_{x \in I_j^n} \frac{|\delta(x)|}{|x-c|^2}
\end{eqnarray*}
where $$|Dq_c(\xi_j^n)| = \frac{|q_c(I_j^n)|}{|I_j^n|}$$
and $\xi_j^n \in I_j^n.$ The a priori bounds gives $K_1 > 0$ such that
$$dist(c, I_j^n) \geq  \frac{1}{K_1} \cdot |I_j^n|.$$
This implies that for some $K > 0$
$$|\eta_j^n| \leq \; K \cdot \sup_{x \in I_j^n} \frac{|\delta(x)|}{|x-c|} \cdot |I_j^n|.$$
Therefore,
\begin{eqnarray*}
\sum_{j=1}^{2^n-1} |\phi_{I_j^n}| & \leq & K \cdot \sum_{j=1}^{2^n-1} \sup_{x
\in I_j^n} \;\frac{|\delta(x)|}{|x-c|} \cdot |I_j^n|\\
& = & K \cdot Z_n
\end{eqnarray*}
Let $\Lambda_n = \cup_{j=0}^{2^n-1} I_j^n$. The a priori bounds imply that there
exists $\tau > 0$ such that
$$|\Lambda_n| \leq (1-\tau) \;|\Lambda_{n-1}|.$$
In particular $|\Lambda| =0$ where $\Lambda \cap \; \Lambda_n$ is the 
Cantor attractor. Now we go back to our estimate and notice
that $Z_n$ is a Riemann sum for
$$\int_{\Lambda_n} \frac{|\delta(x)|}{|x-c|} \;dx.$$
Suppose that $\limsup \; Z_n = Z > 0.$ Let $n \geq 1$ and $m > n$. Then we can
find a Riemann sum $\Sigma_{m,n}$ for
$$\int_{\Lambda_n} \frac{|\delta(x)|}{|x-c|}\;dx$$
by adding positive terms to $Z_m$. Then
$$\int_{\Lambda_n} \frac{|\delta(x)|}{|x-c|} \;dx = \limsup_{m \to \infty}
\;\Sigma_{m,n} \geq \limsup_{m \to \infty} \;Z_m \geq Z > 0.$$
Hence,
$$\int_{\Lambda} \frac{|\delta(x)|}{|x-c|}\;dx\; \geq Z >0.$$
This is impossible because $|\Lambda | = 0$.
Thus we proved
$$\sum_{j=1}^{2^n-1} |\phi_{I_j^n}| \longrightarrow 0.$$
\end{proof}
\end{prop}

\section{Approximation of $R^nf$ by a polynomial map}

The following Lemma is a variation on Sandwich Lemma from \cite{M}.

\begin{lem}(Sandwich) \label{sandwich}
For every $K > 0$ there exists constant $B > 0$ such that the following holds.
Let $\psi_1, \psi_2$ be the compositions of finitely many $\phi, \phi_j \in \text{Diff}_+^2 \;([0, 1]), 1
\leq j \leq n$;
$$\psi_1 = \phi_n \circ \dots \circ \phi_t \circ \dots \phi_1$$
and
$$ \psi_2 = \phi_n \circ \dots \circ \phi_{t+1} \circ \phi \circ \phi_t \circ \dots \phi_1.$$
If $$\displaystyle{\sum_{j}} |\phi_j | + |\phi| \leq K$$
then
$$|\psi_1-\psi_2|_{1} \leq B \;|\phi|.$$
\end{lem}
\begin{proof}
Let $x \in [0, 1]$. For $1 \leq  j \leq n$
let
$$x_j = \phi_{j-1} \circ \dots \circ \phi_2 \circ \phi_1(x) $$
and
\begin{eqnarray*}
D_j = \left(\phi_{j-1} \circ \dots \circ \phi_2 \circ \phi_1 \right)^{'}(x).
\end{eqnarray*}
Furthermore, for $t+1 \leq j \leq n$, let
$$x_j^{\prime} = \phi_{j-1} \circ \dots \circ \phi_{t+1} (\phi(x_{t+1}))$$
and
$$D_j^{\prime} = \left( \phi_{j-1} \circ \dots \circ \phi_{t+1} \right)^{'} (x_{t+1}^{\prime}) \;
\phi^{'} (x_{t+1}) \;D_{t+1}.$$
Now we estimate the difference of the derivatives of $\psi_1, \psi_2$. Namely,
\begin{eqnarray*}
\frac{D\psi_2(x)}{D\psi_1(x)} & = & D\phi(x_{t+1}) \cdot \prod_{j \geq t+1}
\frac{D\phi_j(x_j^{\prime})} {D\phi_j(x_j)}.
\end{eqnarray*}
In the following estimates we will repeatedly apply Lemma $10.3$ from \cite{M} which says,
$$e^{-|\psi|} \leq |D\psi|_{0} \leq e^{|\psi|}.$$
This allows us to get an estimate on $|D\psi_1 - D\psi_2|_{0}$ in terms of
$\displaystyle{\frac{D\psi_2}{D\psi_1}}$.
Now
$$D\phi_j(x_j^{\prime}) = D\phi_j(x_j) +D^2\phi_j(\zeta_j) \;(x_j^{\prime}-x_j).$$
Therefore,
\begin{eqnarray*}
\frac{D\phi_j(x_j^{\prime})}{D\phi_j(x_j)} & \leq &
1+\frac{|D^2\phi_j|_{0}}{D\phi_j(x_j)} \cdot |x_j^{\prime}-x_j| \\
& = & 1+ O (\phi_j) \cdot |x_j^{\prime}-x_j|
\end{eqnarray*}
To continue, we have to estimate $|x_j^{\prime}-x_j|$.
Apply Lemma $10.2$ from \cite{M} to get
\begin{eqnarray*}
|x_j^{\prime}-x_j| & = & O \left(|x_{t+1}^{\prime}-x_{t+1}|\right)\\
& = & O (|\phi|).
\end{eqnarray*}
Because $\sum |\phi_j| + |\phi| \leq K$ there exists $K_1 > 0$ such that
\begin{eqnarray*}
\frac{D\psi_2(x)} {D\psi_1(x)} & \leq & e^{|\phi|}\; \prod_{j \geq t+1} \left( 1 +
O (|\phi_j|\;|\phi|) \right) \\
& \leq & e^{|\phi|} \;e^{K_1 \cdot \sum |\phi_j| \;|\phi|}
\end{eqnarray*}
Hence,
$$\frac{D\psi_2}{D\psi_1} \leq e^{|\phi| (1+K_1 \cdot K)}.$$
We get a lower bound in similar way.
So there exists $K_2 > 0$ such that
$$e^{-K_2 \cdot |\phi|} \leq \frac{|D\psi_2|}{|D\psi_1|} \leq
e^{K_2 \cdot |\phi|}.$$
Finally, there exists $B > 0$ such that
$$|D\psi_2(x) - D\psi_1(x)| \leq B \;|\phi|.$$
\end{proof}

Let $f$ be an infinitely renormalizable $C^{2+|\cdot|}$ unimodal map.
\begin{lem}  \label{q-bound}
There exists $K > 0$ such that for all $n \geq 1$ the following holds
$$\sum_{1 \;\leq \;j \;\leq 2^n-1} |q_j^n| \leq K.$$
\end{lem}
\begin{proof}
The non-linearity norm of $q_j^n$, $j = 1, \dots,2^{n}-1$, is
$$|q_j^n| = \frac{|I_j^n|}{dist\;(I_j^n, c)}.$$
Let
$$Q_n = \sum_{j=1}^{2^n-1}|q_j^n|.$$
Observe that there exists $\tau > 0$ such that for $j=1,2, \dots,2^n-1$
\begin{eqnarray*}
|q_j^{n+1}|+|q_{j+2^n}^{n+1}| & \leq &
\frac{|I_j^{n+1}|+|I_{j+2^n}^{n+1}|}{dist\;(I_j^n,c)}\\
& = & |q_j^n| \; \frac{|I_j^{n+1}|+|I_{j+2^n}^{n+1}|}{|I_j^n|} \\
& = & |q_j^n| \; \frac{|I_j^n-G_j^n|}{|I_j^n|} \leq |q_j^n| (1-\tau).
\end{eqnarray*}
Therefore
$$Q_{n+1} \leq (1-\tau) \;Q_n +|q_{2^n}^{n+1}|.$$
From the a priori bounds we get a constant $K_1 > 0$ such that
$$|q_{2^n}^{n+1}| \leq \frac{|I_{2^n}^{n+1}|}{|G_{2^n}^n|} \leq K_1.$$
Thus
$$Q_{n+1} \leq (1-\tau) Q_n + K_1.$$
This implies the Lemma.
\end{proof}

Consider the map $f: I_0^n \rightarrow I_1^n$, and rescaled affinely range and domain to obtain the 
unimodal map
$$\hat{f_n}: [0, 1] \rightarrow [0, 1].$$
Apply Proposition $\ref{non-lin-bound}$ to obtain the following representation of $\hat{f_n}.$ There exists $c_n \in (0,
1)$ and diffeomorphisms $\phi_{\pm}^n : [0, 1] \rightarrow [0, 1]$ such that
$$\hat{f_n}(x) = \phi_{+}^n \circ q_{c_n}(x), \qquad x\in [c_n, 1]$$
and
$$\hat{f_n}(x) = \phi_{-}^n \circ q_{c_n}(x), \qquad x\in [0, c_n].$$
Furthermore $$|\phi_{\pm}^n| \to 0$$
when $n \to \infty.$
Let $q_0^n = q_{c_n}$. Use Proposition $\ref{approx-qudratic-prop}$ to obtain the following representation for the
$n^{th}$ renormalization of $f$.
$$R^nf = (\phi_{2^n-1}^n \circ q_{2^n-1}^n ) \circ \dots \circ  (\phi_j^n \circ q_j^n ) \circ \dots
\circ ( \phi_1^n \circ q_1^n)  \circ  \phi_{\pm}^n \circ q_0^n.$$
Inspired by \cite{AMM} we introduce the unimodal map
$$f_n = q_{2^n-1}^n \circ \dots \circ q_j^n \circ \dots \circ q_1^n \circ q_0^n.$$

\begin{prop} \label{rnf_close_fn}
If $f$ is an infinitely renormalizable $C^{2+|\cdot|}$ map then
$$\lim_{n \to \infty} |R^nf-f_n|_{1} = 0.$$
\end{prop}
\begin{proof}
Define the diffeomorphisms
$$\psi_{j}^{\pm} =  q_{2^n-1}^n \circ \dots \circ q_j^n \circ (\phi_{j-1}^n \circ q_{j-1}^n) \circ \dots \circ (\phi_1^n
\circ q_1^n) \circ \phi_{\pm}^n$$
with $j = 0,1,2,\dots 2^n.$ Notice that
$$R^nf(x) = \psi_{2^n}^{\pm} \circ q_0^n(x)$$
and that $$f_n(x) = \psi_0^{\pm} \circ  q_0^n(x).$$
where we use again the $\pm$ distinction for points $x \in [0, c_n]$ and $x \in [c_n, 1]$. Apply the Sandwich Lemma
$\ref{sandwich}$ to get a constant $B > 0$ such that
$$ |\psi_{j+1}^{\pm}-\psi_j^{\pm}|_{1} \leq B \cdot |\phi_{j}^n|$$
for $j \geq 1$, and also notice that 
$$|\psi_1^{\pm}-\psi_0^{\pm}|_{1} \leq B \cdot |\phi_{\pm}^n| \longrightarrow 0.$$
We can now apply Proposition $\ref{approx-qudratic-prop}$ to get
$$\lim_{n \to \infty} |\psi_{2^n}^{\pm}-\psi_{0}^{\pm}|_{1} \leq \lim_{n \to \infty} \; B \cdot \sum_{1 \;\leq \;j\;\leq
2^n-1}
|\phi_j^n|+|\phi_{\pm}^n| =0,$$
which implies that: 
$$ \lim_{n \to \infty} |R^nf-f_n|_{1} =0.$$
\end{proof}

\section{Convergence}
Fix an infinitely renormalizable $C^{2+|\cdot|}$ map $f$.
\begin{lem}
For every $N_0 \geq 1$, there exists $n_1 \geq 1$ such that
$f_n$  is $N_0$ times renormalizable whenever $n \geq n_1.$
\end{lem}
\begin{proof}
The a priori bounds from Proposition $\ref{thm-ap}$ gives $d > 0$ such that for $n \geq 1$
$$|(R^nf)^i(c) -(R^nf)^j(c)| \geq d$$
for all $i,j \leq 2^{N_0+1}$ and $i \neq j.$
Now by taking $n$ large enough and using Proposition $\ref{rnf_close_fn}$ we find
$$|f_n^i(c) - f_n^j(c)| \geq \frac{1}{2} d$$
for $i \neq j$ and $i,j \leq 2^{N_0+1}.$
The {\it kneading sequence} of $f_n$ (i.e., the sequence of signs of the derivatives of that 
function) coincides with the kneading sequence of $R^nf$ for
at least $2^{N_0+1}$ positions. We proved
that $f_n$ is $N_0$ times renormalizable because $R^nf$ is $N_0$ times renormalizable.
\end{proof}

The polynomial unimodal maps $f_n$ are in a compact family of quadratic like maps.
This follows from Lemma $\ref{q-bound}.$
The unimodal renormalization theory presented in \cite{Ly} gives us the following.

\begin{prop} \label{ml}
There exists $N_0 \geq 1$ and $n_0 \geq 1$ such that $f_n$ is $N_0$ renormalizable and
$$dist_{1} \;(R^{N_0}f_n, \; W^u) \leq \frac{1}{3} \cdot dist_1 \;(f_n, \;W^u).$$
\end{prop}
Here, $W^u$ is the unstable manifold of the renormalization fixed point contained in the space of quadratic like maps.
Recall that $dist_1$ stands for the $C^1$ distance.

\begin{lem}
There exists $K > 0$ such that for $n \geq 1$
$$dist_1 \; (R^nf, \;W^u) \leq K.$$
\end{lem}
\begin{proof}
This follows from Lemma $\ref{q-bound}$ and Proposition $\ref{rnf_close_fn}$.
\end{proof}

Let $f_{*}^{\omega} \in W^u$ be the analytic renormalization fixed point.

\begin{thm}
If $f$ is an infinitely renormalizable $C^{2+|\cdot|}$ unimodal map. Then
$$\lim_{n \to \infty} dist_0 \left( R^nf, \;f_{*}^{\omega} \right) = 0.$$
\end{thm}
\begin{proof}
For every $K >0$, there exists $A >0$ such that the following holds. Let $f, g$ be renormalizable unimodal maps with
$$|Df|_0, \; |Dg|_0 \leq K$$
then
\begin{eqnarray}
dist_0 (Rf, \;Rg) \;\leq \;A \cdot dist_0 (f, g).  \label{distfg}
\end{eqnarray}
Let $N_0 \geq 1$ be as in Proposition $\ref{ml}$. Now
\begin{eqnarray*}
dist_0 (R^{n+N_0}f, W^u) & \leq  & dist_0 \left( R^{N_0}(R^nf), \; R^{N_0}f_n \right) + dist_0 \left( R^{N_0}f_n, \; W^u
\right) \\
& \leq & A^{N_0} \cdot dist_0 \left( R^nf, \;f_n \right) + \frac{1}{3} \;dist_0
\left(f_n,\;W^u \right)
\end{eqnarray*}
Notice,
$$dist_0(f_n,\; W^u) \leq  dist_0 (f_n, \;R^nf) + dist_0 (R^nf, \;W^u).$$
Thus there exists $K > 0$,
$$dist_0 (R^{n+N_0}f, \;W^u) \leq \frac{1}{3} \;dist_0 (R^nf,\;W^u) + K \cdot dist_0 (R^nf,\;f_n).$$
Let
$$z_n = dist_0 (R^{n \cdot N_0}f, \;W^u)$$
and
$$\delta_n = dist_0(R^nf, \;f_n).$$
Then $$z_{n+1} \leq \frac{1}{3} z_n + K \cdot \delta_{n \cdot N_0}.$$
This implies
$$z_n \leq \sum_{j < n} K \cdot \delta_{j \cdot N_0} \cdot (\frac{1}{3})^{n-j}.$$
Now we use that $\delta_n \to 0$, see Proposition $\ref{rnf_close_fn}$, to get $z_n
\to 0$. So we proved that $R^{n \cdot N_0}f
\; \text{converges to}\; W^u.$ Use $(\ref{distfg})$ and $R(W^u) \subset W^u$ to get that
 $R^nf$ converges to $W^u$ in $C^{0}$ sense. Notice that any limit
of $R^nf$ is infinitely renormalizable. The only infinitely renormalizable map  in
$W^u$ is the fixed point $f_{*}^{\omega}.$
Thus
$$\lim_{n \to \infty} dist_0 \left( R^nf, \; f_{*}^{\omega} \right) = 0.$$
\end{proof}

\section{Slow convergence}

\begin{thm}
Let $d_n > 0$ be any sequence with $d_n \to 0$. There exists an infinitely renormalizable $C^2$ map $f$ with quadratic tip
such that
$$dist_0 \left( R^nf, f^{\omega}_{*} \right)  \geq d_n.$$
\end{thm}
The proof needs some preparation. Use the representation
$$f_{*}^{\omega} =  \phi \circ q_c$$
where $\phi$ is an analytic diffeomorphism. The renormalization domains are denoted by $I_0^n$ with
$$c = \cap_{n \geq 1} I_0^n.$$
Each $I_0^n$ contains two intervals of the $(n+1)^{th}$ generation. Namely $I_0^{n+1}$ and $I_{2^n}^{n+1}.$
Let $$G_n = I_0^n \setminus \left( I_0^{n+1} \cup I_{2^n}^{n+1} \right),$$
$$\hat{G}_n = q_c(G_n) \subset \hat{I}_0^n = q_c(I_0^n)$$
and $\hat{I}_{2^n}^{n+1}  = q_c(I_{2^n}^{n+1}).$ The invariant Cantor set of $f_{*}^{\omega}$ is denoted by $\Lambda$.
Notice,
$$q_c(\Lambda) \cap \hat{I}_0^n \subset \left( \hat{I}_0^{n+1} \cup \hat{I}_{2^n}^{n+1} \right).$$
The gap $\hat{G}_n$ in $\hat{I}_0^n$ does not intersect with $\Lambda$. Choose a family of $C^2$ diffeomorphisms
$$\phi_t : [0, 1] \rightarrow [0, 1]$$ with
\begin{enumerate}
\item [(i)] $D\phi_t(0) = D\phi_t(1) = 1.$
\item [(ii)] $D^2\phi_t(0) = D^2\phi(1) = 0.$
\item [(iii)] For some $C_1 > 0$
$$dist_0 \;(\phi_t, id)  \geq C_1 \cdot t.$$
\item[(iv)] For some $C_2 > 0$
$$|\eta_{\phi_t}|_0 \leq C_2 \cdot t.$$
\end{enumerate}
Let $m = min \; D\phi$ and $t_n = \frac{1}{m \;C_1 \; |\hat{G}_1|} d_n$. Now we will introduce a perturbation
$\tilde{\phi}$ of
$\phi$. Let
$$1_n : [0, 1] \rightarrow \hat{G}_n$$
be the affine orientation preserving homeomorphism. Define
$$\psi: [0, 1] \rightarrow [0, 1]$$ as follows
\begin{displaymath}
\psi(x) = \left\{ \begin{array}{ll}
x & x \notin \cup_{n \geq 0} \hat{G}_n \\
1_n \circ \phi_{t_n} \circ 1_n^{-1}(x)  & x \in \hat{G}_n.
\end{array} \right.
\end{displaymath}
 Let $$f = \phi \circ \psi \circ q_c = \tilde{\phi} \circ q_c.$$
Then $f$ is unimodal map with quadratic tip which is infinitely renormalizable and still has $\Lambda$ as its
invariant Cantor set. This follows from the fact that the perturbation did not affect the critical orbit and it is
located in the complement of the Cantor set. In particular the invariant
Cantor set of $R^nf$ is again $\Lambda \subset I_0^1 \cup I_1^1$ and $G_1$ is the gap of $R^nf$. Notice, by using that
$f_{*}^{\omega}$ is the fixed point of renormalization that for $x \in G_1$
$$R^nf(x) = \phi \circ 1_1 \circ \phi_{t_n} \circ 1_1^{-1} \circ q_c(x)$$
Hence,
\begin{eqnarray*}
|R^nf - f_{*}^{\omega}|_{0} & \geq & \max_{x \in \hat{G}_1} |R^nf(x) - f_{*}^{\omega}(x)| \\
 & \geq & \max_{x \in \hat{G}_1} \; m \cdot |\left( 1_1 \circ \phi_{t_n} \circ 1_1^{-1} \right) q_c(x) -q_c(x)|\\
& \geq & m \cdot \max_{x \in \hat{G}_1} |\left( 1_1 \circ \phi_{t_n} \circ 1_1^{-1} \right) (x) -x|\\
& = & m \cdot |\hat{G}_1| \cdot |\phi_{t_n}-id|_0 \\
& \geq & m \cdot |\hat{G}_1| \cdot C_1 \cdot t_n = d_n.
\end{eqnarray*}
It remains to prove that $f$ is $C^2$. The map $f$ is $C^2$ on $[0, 1] \setminus \left \{c \right \}$ because
$f = \tilde{\phi} \circ q_c$
with $\tilde{\phi} = \phi \circ \psi.$ Where $\phi$ is analytic diffeomorphism and $\psi$ is by construction $C^2$ on
$[0, 1)$.
Notice that, from $(\ref{non-lin-bd-eqn2})$ we have,
\begin{eqnarray}
D^2f(x) & = & 4 \cdot \frac{(x-c)^2}{(1-c)^4} \cdot D^2\tilde{\phi} \left( q_c(x) \right) \label{eqn-in-slow} \\
& - & 2 \cdot \frac{1}{(1-c)^2} \cdot D\tilde{\phi} \left( q_c(x) \right). \nonumber
\end{eqnarray}
We will analyze the above two terms separately. Observe
\begin{displaymath}
D\psi(x) = \left\{ \begin{array}{ll}
1, & x \notin \cup_{n \geq 0} \hat{G}_n \\
|D\phi_{t_n} \left( 1_n^{-1} (x) \right) |,  & x \in \hat{G}_n.
\end{array} \right.
\end{displaymath}
This implies for $x \in G_n$
\begin{eqnarray*}
D\tilde{\phi} \left( q_c(x) \right) & = &  D\phi \left( \psi \circ q_c \right) \cdot D\psi (q_c(x))\\
& = & D\phi(1) \cdot \left(1+O(\hat{I}_0^n) \right) \cdot (1+O (t_n))
\end{eqnarray*}
For $x \notin \cup_{n \geq 1} G_n$ we have
$$D\tilde{\phi}(q_c(x)) = D\phi (q_c(x))$$
This implies that the term
$$x \longmapsto -2 \cdot \frac{1}{(1-c)^2} \cdot D\tilde{\phi}(q_c(x))$$
extends continuously to the whole domain. The first term in $(\ref{eqn-in-slow})$ needs more care. Observe, for $u \in
\hat{G}_n$,
\begin{eqnarray*}
D^2\tilde{\phi}(u) & = & D^2 \phi (\psi(u)) \cdot (D\psi(u))^2 + D\phi(\psi(u)) \cdot 
D^2\psi(u) \\
& = & D^2 \phi(1) \cdot \left(1+O(\hat{I}_0^n) \right) \cdot (1+O(t_n)) + \\ 
&  & D\phi(1) \cdot \left(1+O(\hat{I}_0^n) \right) \cdot (1+O(t_n)) \cdot D^2 \psi(u) \\
& = & D^2\phi(1) \cdot \left(1+O(\hat{I}_0^n) \right) \cdot (1+O(t_n)) + \\
&  & D\phi(1) \cdot \left(1+O(\hat{I}_0^n) \right) \cdot (1+O(t_n)) \cdot 
\frac{1}{|\hat{G}_n|} \cdot O(t_n).
\end{eqnarray*}
This implies that
\begin{displaymath}
4 \; \frac{(x-c)^2} {(1-c)^4} \cdot D^2 \tilde{\phi} (q_c(x))  = \left\{ \begin{array}{ll}
O \left( (x-c)^2 \right) +O(t_n), & x \in \hat{G}_n  \\
O \left( (x-c)^2 \right), & x \notin \cup_{n \geq 0} \hat{G}_n
\end{array} \right.
\end{displaymath}
In particular, the first term of $D^2f$
$$x \longmapsto 4 \; \frac{(x-c)^2} {(1-c)^4} \cdot D^2 \tilde{\phi} (q_c(x))$$
also extends to a continuous function on $[0, 1].$ Indeed, $f$ is $C^2$.

\begin{rem}\label{finalrem}
If the sequence $d_n$ is not summable (and in particular not exponential decaying) then the example
constructed above is not $C^{2+|\cdot|}$. This follows from
$$
\int_{\hat{G}_n}|\eta_{\tilde{\phi}}(x)| dx \asymp t_n.
$$
Thus
$$
\int |\eta_{\tilde{\phi}}|  \asymp \sum d_n=\infty.
$$
Now, equation \ref{inteta} implies that $f$ is not $C^{2+|\cdot|}$.
However, this construction show that in the space of
$C^{2+|\cdot|}$ unimodal maps there are examples whose renormalizations
converges only polynomially. The renormalization fixed point  is
not hyperbolic in the space of $C^{2+|\cdot|}$ unimodal maps.
\end{rem}

\end{document}